\documentclass[11pt]{article}
\usepackage{amssymb,amsfonts,amsmath,amsthm,array,multirow,cite}
\usepackage{enumerate}
\usepackage{cool}
\usepackage{tikz}
\usetikzlibrary{calc,decorations.pathreplacing}

\allowdisplaybreaks[4]

\theoremstyle{plain}
\newtheorem{theorem}{Theorem}[section]
\newtheorem{lemma}[theorem]{Lemma}
\newtheorem{corollary}[theorem]{Corollary}

\theoremstyle{definition}

\newtheorem{example}[theorem]{Example}
\newtheorem{conjecture}[theorem]{Conjecture}

\theoremstyle{remark}
\newtheorem{remark}[theorem]{Remark}

\makeatletter \@addtoreset{equation}{section}
\begin{document}

\begin{center}
{\Large {\bf Euler's partition theorem and lecture hall partition theorem}}

\vskip 6mm

{\small Andrew Y.Z. Wang$^{1}$ and Lei Zhang$^2$\\[2mm]

School of Mathematical Sciences\\
University of Electronic Science and Technology of China\\
Chengdu 611731, P.R. China\\[3mm]

$^1$yzwang@uestc.edu.cn and $^2$zhlmath@outlook.com\\[2mm]}
\end{center}
\noindent {\bf Abstract.} Euler's partition theorem was an important step in the establishment of Integer Partitions as a field in its own right. It has attracted many great mathematicians' attention, including Sylvester, and has foreshadowed a number of further results. Another deep result in Partition Theory is the beautiful Lecture Hall Partition Theorem, which is a finite version of Euler's identity. In this work, we introduce the lecture hall length to partitions into distinct parts, and devise a bijection from partitions into odd parts to partitions into distinct parts, which unifies many results together. Various recent theorems regarding the minima excludant and block index become immediate consequences of this bijection. We also pose an interesting conjecture which may uncover a deep relation between Euler's partition theorem and lecture hall partition theorem.



\section{Introduction}

A partition $\lambda$ of $n$ is a finite weakly decreasing sequence of positive integers $\lambda=(\lambda_1,\lambda_2,\ldots,\lambda_l)$ such that $\lambda_1+\lambda_2+\cdots+\lambda_l=n$. We write $|\lambda|=n$ if $\lambda$ is a partition of $n$. The terms $\lambda_i$ are called the parts of $\lambda$, and the number of parts of $\lambda$ is called the length of $\lambda$, denoted $\ell(\lambda)$. Usually a partition $\lambda$ can also be represented in the following compact form \[\lambda=(\lambda_1^{f_1},\lambda_2^{f_2},\ldots,\lambda_k^{f_k}),\]
where $\lambda_1>\lambda_2>\cdots>\lambda_k$, $f_i$ is the frequency of $\lambda_i$ in $\lambda$, and the superscript $f_i$ may be omitted if it equals one. We use $g(\lambda)$ to denote the largest part of $\lambda$, i.e., $g(\lambda)=\lambda_1$.

There are many results in the theory of partitions that assert the equicardinality of two classes of partitions.
The quintessential example is Euler's celebrated partition identity \cite{Euler48}.
\begin{theorem}
The partitions of $n$ into odd parts are equinumerous with the partitions of $n$ into distinct parts.
\end{theorem}

Euler's partition theorem can be considered as a gem of partition theory. The significance of this theorem lies not only in its simple statement, but also in its beautiful bijective proofs given by Sylvester \cite{SF82} and Glaisher \cite{Glai83}. A description of Sylvester's bijection, also called fish-hook bijection, can be found in \cite[pp. 89--90]{Andrews04}. Glaisher's bijection uses the merging/splitting technique, see \cite[pp. 8--9]{Andrews04} for details.

For convenience, a partition $\lambda$ is called an odd partition if each part is odd, and a distinct partition if all parts are distinct. In the sequel, we use $\mathcal{O}$ and $\mathcal{D}$ to denote the set of odd and distinct partitions, respectively.

As observed by Bessenrodt \cite{Bess91}, Sylvester's bijection implies a refinement of Euler's result. Let $\ell_a(\lambda)$ denote the alternating sum of $\lambda=(\lambda_1,\lambda_2,\ldots,\lambda_l)$, namely,
\[\ell_a(\lambda)=\lambda_1-\lambda_2+\lambda_{3}-\lambda_{4}+\cdots+(-1)^{l-1}\lambda_{l}.\]
\begin{theorem}\label{thmSylv}
The number of odd partitions of $n$ into $l$ parts equals the number of distinct partitions of $n$ whose
alternating sum is $l$. In terms of generating functions, we have
\begin{align*}
\sum\limits_{\lambda\in\mathcal{O}}x^{\ell(\lambda)}q^{|\lambda|}
=\sum\limits_{\pi\in\mathcal{D}}x^{\ell_a(\pi)}q^{|\pi|}.
\end{align*}
\end{theorem}

Coincidentally, Glaisher's bijection also leads to a simple refinement of Euler's theorem. Let $o(\lambda)$ be the number of different odd parts of $\lambda$ occurring an odd number of times, and let $n_o(\lambda)$ be the number of odd parts in $\lambda$.
\begin{theorem}\label{thmGla}
The number of odd partitions of $n$ with exactly $m$ different parts occurring an odd number of times is equal to the number of distinct partitions of $n$ with $m$ odd parts. In the notation of generating functions, we have
\begin{align*}
\sum\limits_{\lambda\in\mathcal{O}}y^{o(\lambda)}q^{|\lambda|}
=\sum\limits_{\pi\in\mathcal{D}}y^{n_o(\pi)}q^{|\pi|}.
\end{align*}
\end{theorem}

We have more to say about these two results. Theorem \ref{thmSylv} tells us that the distribution of the length on odd partitions is the same as the distribution of the alternating sum on distinct partitions. It is natural to consider which statistic of odd partitions correspond to the length of distinct partitions. To answer this question, Li, Wang and Xu \cite{LWX23} introduced a new partition statistic, called the block odd index, defined in the next paragraph.

Given an odd partition $\lambda$, we group the parts of $\lambda$, working from the largest to the smallest, into blocks each of which contains at most two distinct odd integers. Assuming that $2t-1$ is the largest part. If $2t-3$ appears in $\lambda$, then all the parts of size $2t-1$ and $2t-3$ form a block containing two distinct integers; otherwise, all the parts of size $2t-1$ inhabit a block with only one distinct integer. Remove this block from $\lambda$ and iterate the same operation on the remaining parts. Continue this process until eventually every part of $\lambda$ lies in a unique block. The weight of a block is defined to be $1$ if it is a block with only $1$'s, and $2$ otherwise. We define the block odd index of $\lambda$, denoted $b_o(\lambda)$, to be the sum of the weight of its blocks. For instance, if
\[\lambda=(25^2,23^7,21^5,17^4,15,11,9^6,5^3,3,1^6),\]
then it has $6$ blocks, namely $(25^2,23^7),(21^5),(17^4,15),(11,9^6),(5^3,3)$ and $(1^6)$. Thus,
\[b_o(\lambda)=2+2+2+2+2+1=11.\]
Letting $\pi=(23^4,17^2,15^5,11^3,7,5^6,3^5,1^2)$, then the blocks of $\pi$ are $(23^4),(17^2,15^5)$, $(11^3)$, $(7,5^6)$ and $(3^5,1^2)$. So $b_o(\pi)=2+2+2+2+2=10$.

Li, Wang and Xu \cite{LWX23} obtained the following strong result.
\begin{theorem}\label{thmEulerstr}
The number of odd partitions of $n$ into $l$ parts with $m$ distinct parts occurring an odd number of times and block odd index $k$ equals the number of distinct partitions of $n$ into $k$ parts with $m$ odd parts and alternating sum $l$. Equivalently, we have
\begin{align*}
\sum_{\lambda\in\mathcal{O}}x^{\ell(\lambda)}y^{o(\lambda)}z^{b_o(\lambda)}q^{|\lambda|}
=\sum_{\pi\in\mathcal{D}}x^{\ell_a(\pi)}y^{n_{o}(\pi)}z^{\ell(\pi)}q^{|\pi|}.
\end{align*}
\end{theorem}

Given a partition $\lambda$, let $o_{1,4}(\lambda)$ and $o_{3,4}(\lambda)$ denote the number of distinct parts which occur an odd number of times, and are congruent to $1$ and $3$ modulo $4$, respectively. Obviously, $o(\lambda)=o_{1,4}(\lambda)+o_{3,4}(\lambda)$. For a partition $\pi=(\pi_1,\pi_2,\ldots,\pi_l)$, we call each $\pi_{2i-1}$ an odd indexed part, and each $\pi_{2i}$ an even indexed part. Namely, whether $\pi_i$ is an odd or even indexed part depends on the parity of the subscript $i$.
Let $n_{o,o}(\pi)$ and $n_{o,e}(\pi)$ denote the number of odd and even indexed odd parts of $\pi$, respectively. It is clear that $n_o(\pi)=n_{o,o}(\pi)+n_{o,e}(\pi)$.

Wang and Zhang \cite{WZh23} strengthened Theorem \ref{thmEulerstr} further.
\begin{theorem}\label{thmfr}
The number of odd partitions of $n$ into $l$ parts with $i$ distinct parts congruent to $1$ modulo $4$ and each occurring an odd number of times, $j$ distinct parts congruent to $3$ modulo $4$ and each occurring an odd number of times, and block odd index $k$ is equal to the number of distinct partitions of $n$ into $k$ parts with alternating sum $l$, $i$ odd indexed odd parts and $j$ even indexed odd parts. In the language of generating functions, we have
\begin{align*}
\sum_{\lambda\in\mathcal{O}}x^{\ell(\lambda)}y_1^{o_{1,4}(\lambda)}y_2^{{o_{3,4}(\lambda)}}z^{b_o(\lambda)}q^{|\lambda|}
=\sum_{\pi\in\mathcal{D}}x^{\ell_a(\pi)}y_1^{n_{o,o}(\pi)}y_2^{n_{o,e}(\pi)}z^{\ell(\pi)}q^{|\pi|}.
\end{align*}
\end{theorem}

There are also several other refinements and variations of Euler's partition theorem. One beautiful result among these is the Lecture Hall Partition Theorem \cite{BE97}. A lecture hall partition of length $N$ is a partition $(\lambda_1,\lambda_2,\ldots,\lambda_N)$ allowing the zero part to appear and satisfying
\[\frac{\lambda_1}{N}\geq\frac{\lambda_2}{N-1}\geq\cdots\geq\frac{\lambda_N}{1}\geq0,\]
which is called the lecture hall condition. Bousquet-M\'{e}lou and Eriksson \cite{BE97} obtained a surprising identity for lecture hall partitions.
\begin{theorem}\label{themLH}
The number of lecture hall partitions of $n$ with length $N$ equals the number of odd partitions of $n$ with largest part at most $2N-1$.
\end{theorem}

Since lecture hall partitions always have distinct nonzero parts, Euler's partition theorem is the limiting case of the lecture hall partition theorem as $N$ tends to infinity. Therefore, we can say that lecture hall partition theorem is a finite version of Euler's partition theorem.

In fact, we can derive a refinement of Euler's identity from Theorem \ref{themLH}. Given a distinct partition $\lambda$, we can always change $\lambda$ to be a lecture hall partition by appending some zero parts to $\lambda$. Moreover, if $\lambda$ is a lecture hall partition of length $N$, it is easy to prove that $\lambda$ is also a lecture hall partition of length $N+1$. For instance, if $\lambda=(5,4,1)$, since
\begin{align*}
&\frac{5}{3}\ngeq\frac{4}{2}\geq\frac{1}{1},\\
&\frac{5}{4}\ngeq\frac{4}{3}\geq\frac{1}{2}\geq\frac{0}{1},\\
&\frac{5}{5}\geq\frac{4}{4}\geq\frac{1}{3}\geq\frac{0}{2}\geq\frac{0}{1},
\end{align*}
we must add at least two zero parts to $\lambda$ to satisfy the lecture hall condition. This inspires us to propose  the notion of lecture hall length of $\lambda$, denoted $\ell_h(\lambda)$, which is defined to be the length of the corresponding lecture hall partition obtained from $\lambda$ by adding the fewest zero parts. For example, if $\lambda=(5,4,1)$ and $\pi=(6,2,1)$, then $\ell_h(\lambda)=5$ and $\ell_h(\pi)=3$. In Section \ref{seclh}, we will give a formula to compute the lecture hall length of a distinct partition.

Now we can state Bousquet-M\'{e}lou and Eriksson's refinement as follows.
\begin{theorem}\label{thmLhr}
The number of odd partitions of $n$ with largest part being $2N-1$ is equal to the number of distinct partitions of $n$ with the lecture hall length being $N$. In the notation of generating functions, we have
\begin{align*}
\sum\limits_{\lambda\in\mathcal{O}}w^{(g(\lambda)+1)/2}q^{|\lambda|}
=\sum\limits_{\pi\in\mathcal{D}}w^{\ell_h(\pi)}q^{|\pi|}.
\end{align*}
\end{theorem}

It is surprising that we can unify Theorem \ref{thmfr} and Theorem \ref{thmLhr} together.
\begin{theorem}\label{thmmain}
The odd partitions of $n$ into $l$ parts with $i$ distinct parts congruent to $1$ modulo $4$ and each occurring an odd number of times, $j$ distinct parts congruent to $3$ modulo $4$ and each occurring an odd number of times, block odd index $k$, and largest part $2N-1$ are equinumerous with the distinct partitions of $n$ into $k$ parts with alternating sum $l$, $i$ odd indexed odd parts and $j$ even indexed odd parts, and lecture hall length $N$.
Equivalently,
\begin{align*}
\sum_{\lambda\in\mathcal{O}}x^{\ell(\lambda)}y_1^{o_{1,4}(\lambda)}y_2^{{o_{3,4}(\lambda)}}z^{b_o(\lambda)}
w^{(g(\lambda)+1)/2}q^{|\lambda|}
=\sum_{\pi\in\mathcal{D}}x^{\ell_a(\pi)}y_1^{n_{o,o}(\pi)}y_2^{n_{o,e}(\pi)}z^{\ell(\pi)}w^{\ell_{h}(\pi)}q^{|\pi|}.
\end{align*}
\end{theorem}

See Figure \ref{figsr} for an illustration of our strengthened refinement.

\begin{figure}[ht]
\begin{center}
\begin{tikzpicture}

\draw (0,0)--(7.6,0);     \draw (0,0.6)--(7.6,0.6);  \draw (0,1.2)--(7.6,1.2);
\draw (0,1.8)--(7.6,1.8); \draw (0,2.4)--(7.6,2.4);  \draw (0,3)--(7.6,3);
\draw (0,3.6)--(7.6,3.6); \draw (0,4.2)--(7.6,4.2);  \draw (0,4.8)--(7.6,4.8);
\draw (0,5.4)--(7.6,5.4); \draw (0,6)--(7.6,6);      \draw (0,6.8)--(7.6,6.8);

\draw (0,0)--(0,6.8); \draw (1.6,0)--(1.6,6.8); \draw (2.6,0)--(2.6,6.8); \draw (4,0)--(4,6.8);
\draw (5.4,0)--(5.4,6.8); \draw (6.4,0)--(6.4,6.8); \draw (7.6,0)--(7.6,6.8);

\node at (0.8,6.4) {$\lambda\in\mathcal{O}$}; \node at (2.1,6.4) {$\ell(\lambda)$};
\node at (3.3,6.4) {$o_{1,4}(\lambda)$}; \node at (4.7,6.4) {$o_{3,4}(\lambda)$};
\node at (5.9,6.4) {$b_o(\lambda)$}; \node at (7,6.4) {$\frac{g(\lambda)+1}{2}$};

\node at (0.8,5.7) {$(9,1)$}; \node at (2.1,5.7) {$2$}; \node at (3.3,5.7) {$2$};
\node at (4.7,5.7) {$0$}; \node at (5.9,5.7) {$3$}; \node at (7,5.7) {$5$};

\node at (0.8,5.1) {$(7,3)$}; \node at (2.1,5.1) {$2$}; \node at (3.3,5.1) {$0$};
\node at (4.7,5.1) {$2$}; \node at (5.9,5.1) {$4$}; \node at (7,5.1) {$4$};

\node at (0.8,4.5) {$(7,1^3)$}; \node at (2.1,4.5) {$4$}; \node at (3.3,4.5) {$1$};
\node at (4.7,4.5) {$1$}; \node at (5.9,4.5) {$3$}; \node at (7,4.5) {$4$};

\node at (0.8,3.9) {$(5^2)$}; \node at (2.1,3.9) {$2$}; \node at (3.3,3.9) {$0$};
\node at (4.7,3.9) {$0$}; \node at (5.9,3.9) {$2$}; \node at (7,3.9) {$3$};

\node at (0.8,3.3) {$(5,3,1^2)$}; \node at (2.1,3.3) {$4$}; \node at (3.3,3.3) {$1$};
\node at (4.7,3.3) {$1$}; \node at (5.9,3.3) {$3$}; \node at (7,3.3) {$3$};

\node at (0.8,2.7) {$(5,1^5)$}; \node at (2.1,2.7) {$6$}; \node at (3.3,2.7) {$2$};
\node at (4.7,2.7) {$0$}; \node at (5.9,2.7) {$3$}; \node at (7,2.7) {$3$};

\node at (0.8,2.1) {$(3^3,1)$}; \node at (2.1,2.1) {$4$}; \node at (3.3,2.1) {$1$};
\node at (4.7,2.1) {$1$}; \node at (5.9,2.1) {$2$}; \node at (7,2.1) {$2$};

\node at (0.8,1.5) {$(3^2,1^4)$}; \node at (2.1,1.5) {$6$}; \node at (3.3,1.5) {$0$};
\node at (4.7,1.5) {$0$}; \node at (5.9,1.5) {$2$}; \node at (7,1.5) {$2$};

\node at (0.8,0.9) {$(3,1^7)$}; \node at (2.1,0.9) {$8$}; \node at (3.3,0.9) {$1$};
\node at (4.7,0.9) {$1$}; \node at (5.9,0.9) {$2$}; \node at (7,0.9) {$2$};

\node at (0.8,0.3) {$(1^{10})$}; \node at (2.1,0.3) {$10$}; \node at (3.3,0.3) {$0$};
\node at (4.7,0.3) {$0$}; \node at (5.9,0.3) {$1$}; \node at (7,0.3) {$1$};


\draw (7.7,0)--(15.4,0);     \draw (7.7,0.6)--(15.4,0.6);  \draw (7.7,1.2)--(15.4,1.2);
\draw (7.7,1.8)--(15.4,1.8); \draw (7.7,2.4)--(15.4,2.4);  \draw (7.7,3)--(15.4,3);
\draw (7.7,3.6)--(15.4,3.6); \draw (7.7,4.2)--(15.4,4.2);  \draw (7.7,4.8)--(15.4,4.8);
\draw (7.7,5.4)--(15.4,5.4); \draw (7.7,6)--(15.4,6);      \draw (7.7,6.8)--(15.4,6.8);

\draw (7.7,0)--(7.7,6.8); \draw (9.5,0)--(9.5,6.8); \draw (10.5,0)--(10.5,6.8); \draw (11.9,0)--(11.9,6.8);
\draw (13.3,0)--(13.3,6.8); \draw (14.3,0)--(14.3,6.8); \draw (15.4,0)--(15.4,6.8);

\node at (8.6,6.4) {$\pi\in\mathcal{D}$}; \node at (10,6.4) {$\ell_a(\pi)$};
\node at (11.2,6.4) {$n_{o,o}(\pi)$}; \node at (12.6,6.4) {$n_{o,e}(\pi)$};
\node at (13.8,6.4) {$\ell(\pi)$}; \node at (14.85,6.4) {$\ell_h(\pi)$};

\node at (8.6,5.7) {$(5,4,1)$}; \node at (10,5.7) {$2$}; \node at (11.2,5.7) {$2$};
\node at (12.6,5.7) {$0$}; \node at (13.8,5.7) {$3$}; \node at (14.85,5.7) {$5$};

\node at (8.6,5.1) {$(4,3,2,1)$}; \node at (10,5.1) {$2$}; \node at (11.2,5.1) {$0$};
\node at (12.6,5.1) {$2$}; \node at (13.8,5.1) {$4$}; \node at (14.85,5.1) {$4$};

\node at (8.6,4.5) {$(5,3,2)$}; \node at (10,4.5) {$4$}; \node at (11.2,4.5) {$1$};
\node at (12.6,4.5) {$1$}; \node at (13.8,4.5) {$3$}; \node at (14.85,4.5) {$4$};

\node at (8.6,3.9) {$(6,4)$}; \node at (10,3.9) {$2$}; \node at (11.2,3.9) {$0$};
\node at (12.6,3.9) {$0$}; \node at (13.8,3.9) {$2$}; \node at (14.85,3.9) {$3$};

\node at (8.6,3.3) {$(6,3,1)$}; \node at (10,3.3) {$4$}; \node at (11.2,3.3) {$1$};
\node at (12.6,3.3) {$1$}; \node at (13.8,3.3) {$3$}; \node at (14.85,3.3) {$3$};

\node at (8.6,2.7) {$(7,2,1)$}; \node at (10,2.7) {$6$}; \node at (11.2,2.7) {$2$};
\node at (12.6,2.7) {$0$}; \node at (13.8,2.7) {$3$}; \node at (14.85,2.7) {$3$};

\node at (8.6,2.1) {$(7,3)$}; \node at (10,2.1) {$4$}; \node at (11.2,2.1) {$1$};
\node at (12.6,2.1) {$1$}; \node at (13.8,2.1) {$2$}; \node at (14.85,2.1) {$2$};

\node at (8.6,1.5) {$(8,2)$}; \node at (10,1.5) {$6$}; \node at (11.2,1.5) {$0$};
\node at (12.6,1.5) {$0$}; \node at (13.8,1.5) {$2$}; \node at (14.85,1.5) {$2$};

\node at (8.6,0.9) {$(9,1)$}; \node at (10,0.9) {$8$}; \node at (11.2,0.9) {$1$};
\node at (12.6,0.9) {$1$}; \node at (13.8,0.9) {$2$}; \node at (14.85,0.9) {$2$};

\node at (8.6,0.3) {$(10)$}; \node at (10,0.3) {$10$}; \node at (11.2,0.3) {$0$};
\node at (12.6,0.3) {$0$}; \node at (13.8,0.3) {$1$}; \node at (14.85,0.3) {$1$};

\end{tikzpicture}\caption{An illustration of Theorem \ref{thmmain}.}\label{figsr}
\end{center}
\end{figure}

\begin{remark}
While Euler's partition theorem and the lecture hall partition theorem admit numerous bijective proofs (see \cite{Sav16} and the references therein), none of these bijections extends to Theorem \ref{thmmain}.
\end{remark}

The rest of the paper is organized as follows. In Section \ref{secmap}, we construct a simple bijection from odd partitions to distinct partitions, which is devised to prove our main theorem combinatorially. The goal of Section \ref{secsta} is to prove the first four pairs of statistics. In Section \ref{seclh}, we present a formula to compute the lecture hall length and discuss its relevant properties. In Section \ref{secchain}, we study the chains of odd partitions and establish some results about the $\delta$-fixed chains. The contribution of Sections \ref{secub} and  \ref{seclb} is to prove the last pair of statistics in our main theorem. Finally, we give some concluding remarks in Section \ref{seccr}.

\section{Bijection}\label{secmap}
In this section, we aim to construct a simple bijection $\varphi$ from $\mathcal{O}$ to $\mathcal{D}$, which will provide us with a bijective proof of Theorem \ref{thmmain}.

Given a partition $\lambda\in\mathcal{O}$, we define an operation $\tau$ on each block of $\lambda$ as follows.
\begin{itemize}
\item For a block of weight $2$, assuming that it is $(t^a,(t-2)^b)$ where $a\geq 1$ and $b\geq0$, we change it
      to be $(t^{a-1},(t-2)^{b+1})$. Namely, we change a part of size $t$ to be a part of size $t-2$.
\item For a block of weight $1$, then it must be $(1^a)$ where $a\geq1$. We change it to be $(1^{a-1})$. Namely,
      we remove a part of size $1$.
\end{itemize}
Let $\tau(\lambda)$ denote the produced new partition. Clearly, we have $|\tau(\lambda)|=|\lambda|-b_o(\lambda)$. Furthermore, $b_o(\tau(\lambda))=b_o(\lambda)-1$ if the last block of $\lambda$ is $(3,1^a)$ or $(1)$; otherwise $b_o(\tau(\lambda))=b_o(\lambda)$.

For example, if $\lambda=(23^5,19^8,17^3,15^3,13^6,9,7^4,5^2,3^5,1)$, we group its parts into blocks
\[(23^5)(19^8,17^3)(15^3,13^6)(9,7^4)(5^2,3^5)(1),\]
and then change its blocks as follows
\begin{align*}
(23^5)&\rightarrow(23^4,21),\\
(19^8,17^3)&\rightarrow(19^7,17^4),\\
(15^3,13^6)&\rightarrow(15^2,13^7),\\
(9,7^4)&\rightarrow(7^5),\\
(5^2,3^5)&\rightarrow(5,3^6),\\
(1)&\rightarrow\emptyset.
\end{align*}
Thus, $\tau(\lambda)=(23^4,21,19^7,17^4,15^2,13^7,7^5,5,3^6)$ and $b_o(\tau(\lambda))=10$.

We now repeat this transformation until the partition is completely emptied. As a result, we produce a nonincreasing sequence $\mathbf{s}(\lambda)=(s_1,s_2,s_3\ldots)$ where
\[s_1=b_o(\lambda), s_2=b_o(\tau(\lambda)), s_3=b_o(\tau^2(\lambda)),\ldots.\]
Note that $\mathbf{s}$ is a gap-free sequence, which has $1$ as its last entry and satisfies $0\leq s_i-s_{i+1}\leq 1$ for all $i\geq1$. In addition, $s_1+s_2+s_3+\cdots=|\lambda|$.

To see the algorithm more clearly, we treat the case where $\lambda=(9^3,7,3^2,1^4)$, and illustrate the transformations as follows
\begin{align*}
(9^3,7,3^2,1^4)
&\stackrel{-4}{\longrightarrow}(9^2,7^2,3,1^5)\stackrel{-4}{\longrightarrow}(9,7^3,1^6)
\stackrel{-3}{\longrightarrow}(7^4,1^5)\stackrel{-3}{\longrightarrow}(7^3,5,1^4)
\stackrel{-3}{\longrightarrow}(7^2,5^2,1^3)\\
&\stackrel{-3}{\longrightarrow}(7,5^3,1^2)\stackrel{-3}{\longrightarrow}(5^4,1)
\stackrel{-3}{\longrightarrow}(5^3,3)\stackrel{-2}{\longrightarrow}(5^2,3^2)
\stackrel{-2}{\longrightarrow}(5,3^3)\stackrel{-2}{\longrightarrow}(3^4)\\
&\stackrel{-2}{\longrightarrow}(3^3,1)\stackrel{-2}{\longrightarrow}(3^2,1^2)
\stackrel{-2}{\longrightarrow}(3,1^3)\stackrel{-2}{\longrightarrow}(1^4)
\stackrel{-1}{\longrightarrow}(1^3)\stackrel{-1}{\longrightarrow}(1^2)\\
&\stackrel{-1}{\longrightarrow}(1)\stackrel{-1}{\longrightarrow}\emptyset,
\end{align*}
where the symbol $\lambda\stackrel{-b_o(\lambda)}{\longrightarrow}$ means that $\lambda$ is reduced by $b_o(\lambda)$. So,
\[\mathbf{s}(\lambda)=(4,4,3,3,3,3,3,3,2,2,2,2,2,2,2,1,1,1,1).\]

Based on $\mathbf{s}(\lambda)$, we can construct a distinct partition $\pi=(\pi_1,\pi_2,\ldots)$ where $\pi_i$ is defined to be the number of entries in $\mathbf{s}(\lambda)$ greater than or equal to $i$. For instance, continuing the previous example, we get $\pi=(19,15,8,2)$. Now the construction of $\varphi$ is finished, and $\pi$ is the desired $\varphi(\lambda)$.

We next show that the procedure $\varphi$ is invertible.

We consider another operation $\lambda\stackrel{+w}{\longrightarrow}$, adding a weight $w$ to a partition $\lambda$ where $w=b_o(\lambda)$ or $b_o(\lambda)+1$. The operation $\lambda\stackrel{+w}{\longrightarrow}$ is defined as follows. We first group the parts of $\lambda$ into blocks, which is very similar to that in the definition of block odd index. But this time we start from the smallest part, and take into account the parity of $w$. More precisely, if $w$ is even, we group the parts in the ordinary manner, and change each block $((t+2)^a,t^b)$ to $((t+2)^{a+1},t^{b-1})$ where $a\geq0$ and $b\geq1$. If $w$ is odd, we let all $1$'s inhabit a block regardless of whether $3$ exists or not, and group the remaining parts in the ordinary manner. We change the block $(1^a)$ to $(1^{a+1})$ where $a\geq 0$, and change the other blocks as above.

For example, let $\lambda=(23^4,21,19^7,17^4,15^2,13^7,7^5,5,3^6)$. Clearly, $b_o(\lambda)=10$. If $w=10$, then we group the parts into blocks
\[(23^4,21)(19^7,17^4)(15^2,13^7)(7^5)(5,3^6),\]
and change the blocks as follows
\begin{align*}
(23^4,21)&\to(23^5),\\
(19^7,17^4)&\to(19^8,17^3),\\
(15^2,13^7)&\to(15^3,13^6),\\
(7^5)&\to(9,7^4),\\
(5,3^6)&\to(5^2,3^5).
\end{align*}
Thus, $\lambda\stackrel{+w}{\longrightarrow}$ is $(23^5,19^8,17^3,15^3,13^6,9,7^4,5^2,3^5)$. If $w=11$, then we group the parts in the same way but add an empty block $(1^0)$ this time, and change the blocks as follows
\begin{align*}
(23^4,21)&\to(23^5),\\
(19^7,17^4)&\to(19^8,17^3),\\
(15^2,13^7)&\to(15^3,13^6),\\
(7^5)&\to(9,7^4),\\
(5,3^6)&\to(5^2,3^5),\\
(1^0)&\to(1).
\end{align*}
Thus, $\lambda\stackrel{+w}{\longrightarrow}$ is $(23^5,19^8,17^3,15^3,13^6,9,7^4,5^2,3^5,1)$.

Given a distinct partition $\pi=(\pi_1,\pi_2,\ldots,\pi_k)$, we produce a gap-free sequence \[\mathbf{s}=(\underbrace{k,k,\ldots,k}_{\pi_k},\underbrace{k-1,k-1,\ldots,k-1}_{\pi_{k-1}-\pi_k},\ldots,
\underbrace{2,2,\ldots,2}_{\pi_2-\pi_3},\underbrace{1,1,\ldots,1}_{\pi_1-\pi_2}).\]
We view each entry of $\mathbf{s}$ as a weight, for which we perform the operation $\lambda\stackrel{+w}{\longrightarrow}$ successively, starting from the empty partition and adding the weights from $1$ to $k$. Finally, we arrive at an odd partition $\lambda$ such that $|\lambda|=|\pi|$. For instance, taking $\pi=(24,20,13,7,3)$, then
\[\mathbf{s}=(5,5,5,4,4,4,4,3,3,3,3,3,3,2,2,2,2,2,2,2,1,1,1,1).\]
The process of adding entries of $\mathbf{s}$ is
\begin{align*}
\emptyset&\stackrel{+1}{\longrightarrow}(1)\stackrel{+1}{\longrightarrow}(1^2)\stackrel{+1}{\longrightarrow}(1^3)
\stackrel{+1}{\longrightarrow}(1^4)\\
&\stackrel{+2}{\longrightarrow}(3,1^3)\stackrel{+2}{\longrightarrow}(3^2,1^2)\stackrel{+2}{\longrightarrow}(3^3,1)
\stackrel{+2}{\longrightarrow}(3^4)\stackrel{+2}{\longrightarrow}(5,3^3)\stackrel{+2}{\longrightarrow}(5^2,3^2)
\stackrel{+2}{\longrightarrow}(5^3,3)\\
&\stackrel{+3}{\longrightarrow}(5^4)(1)\stackrel{+3}{\longrightarrow}(7,5^3)(1^2)
\stackrel{+3}{\longrightarrow}(7^2,5^2)(1^3)\stackrel{+3}{\longrightarrow}(7^3,5)(1^4)
\stackrel{+3}{\longrightarrow}(7^4)(1^5)\stackrel{+3}{\longrightarrow}(9,7^3)(1^6)\\
&\stackrel{+4}{\longrightarrow}(9^2,7^2)(3,1^5)\stackrel{+4}{\longrightarrow}(9^3,7)(3^2,1^4)
\stackrel{+4}{\longrightarrow}(9^4)(3^3,1^3)\stackrel{+4}{\longrightarrow}(11,9^3)(3^4,1^2)\\
&\stackrel{+5}{\longrightarrow}(11^2,9^2)(5,3^3)(1^3)\stackrel{+5}{\longrightarrow}(11^3,9)(5^2,3^2)(1^4)
\stackrel{+5}{\longrightarrow}(11^4)(5^3,3)(1^5).
\end{align*}
Therefore, $\lambda=(11^4,5^3,3,1^5)$.

We can give a visual description of our bijection $\varphi$. Let $\lambda=(\lambda_1,\lambda_2,\ldots)$ be a partition of $n$. The Ferrers graph of $\lambda$ is obtained by drawing a left-justified array of $n$ dots with $\lambda_i$ dots in the $i$-th row. For instance, the Ferrers graph of the partition $(5,5,3,2,2)$ is given by Figure \ref{figFer}.
\begin{figure}[ht]
\begin{center}
\begin{tikzpicture}

\draw[fill] (0,0) circle (0.08cm);  \draw[fill] (0.4,0) circle (0.08cm);

\draw[fill] (0,0.4) circle (0.08cm);  \draw[fill] (0.4,0.4) circle (0.08cm);

\draw[fill] (0,0.8) circle (0.08cm);  \draw[fill] (0.4,0.8) circle (0.08cm); \draw[fill] (0.8,0.8) circle (0.08cm);

\draw[fill] (0,1.2) circle (0.08cm);  \draw[fill] (0.4,1.2) circle (0.08cm); \draw[fill] (0.8,1.2) circle (0.08cm);
\draw[fill] (1.2,1.2) circle (0.08cm);  \draw[fill] (1.6,1.2) circle (0.08cm);

\draw[fill] (0,1.6) circle (0.08cm);  \draw[fill] (0.4,1.6) circle (0.08cm); \draw[fill] (0.8,1.6) circle (0.08cm);
\draw[fill] (1.2,1.6) circle (0.08cm);  \draw[fill] (1.6,1.6) circle (0.08cm);

\end{tikzpicture}\caption{The Ferrers graph of $(5,5,3,2,2)$.}\label{figFer}
\end{center}
\end{figure}
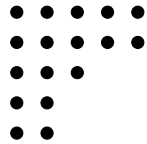
There are various transformations on Ferrers graphs. As an elementary example, if we interchange rows and columns in the Ferrers graph of $\lambda$, we are left with a new Ferrers graph. The corresponding partition is called the conjugate of $\lambda$, denoted $\lambda'$. Equivalently, $\lambda'$ is the partition obtained by reading the Ferrers graph of $\lambda$ from left to right, rather than from top to bottom. In particular, the conjugate of a distinct partition is a gap-free partition (a partition in which every integer less than the largest part must appear at least once).

When we perform the operation $\tau$ on $\lambda$, we in fact reduce $\lambda$ by $b_o(\lambda)$. To record this quantity, we draw a column with $b_o(\lambda)$ dots. Repeat this procedure until the partition is completely emptied. Finally, we get a Ferrers graph, whose columns form a gap-free partition, i.e., it is the Ferrers graph of a distinct partition. See Figure \ref{figbije} for an illustration, where we put the parts of a partition into one column to save space.

\begin{figure}[ht]
\begin{center}
\begin{tikzpicture}

\node at (0,2) {\footnotesize$1^3$};\node at (0,2.4) {\footnotesize$7$};
\node at (0,2.8) {\footnotesize$13$};

\draw[fill] (0.45,0) circle (0.08cm);  \draw[fill] (0.45,0.4) circle (0.08cm);
\draw[fill] (0.45,0.8) circle (0.08cm);
\draw[fill] (0.45,1.2) circle (0.08cm); \draw[fill] (0.45,1.6) circle (0.08cm);

\draw[->] (0.25,2.4)--(0.65,2.4); \node at (0.45,2.6) {\footnotesize$-5$};

\draw[->] (0.75,0.8)--(1.15,0.8);
\node at (0.95,2) {\footnotesize$1^2$};\node at (0.95,2.4) {\footnotesize$5$};
\node at (0.95,2.8) {\footnotesize$11$};

\draw[fill] (1.45,0) circle (0.08cm);  \draw[fill] (1.45,0.4) circle (0.08cm);
\draw[fill] (1.45,0.8) circle (0.08cm);
\draw[fill] (1.45,1.2) circle (0.08cm); \draw[fill] (1.45,1.6) circle (0.08cm);

\draw[fill] (1.85,0) circle (0.08cm);  \draw[fill] (1.85,0.4) circle (0.08cm);
\draw[fill] (1.85,0.8) circle (0.08cm);
\draw[fill] (1.85,1.2) circle (0.08cm); \draw[fill] (1.85,1.6) circle (0.08cm);

\draw[->] (1.25,2.4)--(2.05,2.4); \node at (1.65,2.6) {\footnotesize$-5$};

\draw[->] (2.15,0.8)--(2.55,0.8);
\node at (2.35,2) {\footnotesize$1$};\node at (2.35,2.4) {\footnotesize$3$};
\node at (2.35,2.8) {\footnotesize$9$};

\draw[fill] (2.85,0) circle (0.08cm);  \draw[fill] (2.85,0.4) circle (0.08cm);
\draw[fill] (2.85,0.8) circle (0.08cm);
\draw[fill] (2.85,1.2) circle (0.08cm); \draw[fill] (2.85,1.6) circle (0.08cm);

\draw[fill] (3.25,0) circle (0.08cm);  \draw[fill] (3.25,0.4) circle (0.08cm);
\draw[fill] (3.25,0.8) circle (0.08cm);
\draw[fill] (3.25,1.2) circle (0.08cm); \draw[fill] (3.25,1.6) circle (0.08cm);

\draw[fill] (3.65,0.4) circle (0.08cm); \draw[fill] (3.65,0.8) circle (0.08cm);
\draw[fill] (3.65,1.2) circle (0.08cm); \draw[fill] (3.65,1.6) circle (0.08cm);

\draw[->] (2.65,2.4)--(3.85,2.4); \node at (3.25,2.6) {\footnotesize$-4$};

\draw[->] (3.95,0.8)--(4.35,0.8);
\node at (4.15,2.2) {\footnotesize$1^2$};\node at (4.15,2.6) {\footnotesize$7$};

\draw[fill] (4.65,0) circle (0.08cm);  \draw[fill] (4.65,0.4) circle (0.08cm);
\draw[fill] (4.65,0.8) circle (0.08cm);
\draw[fill] (4.65,1.2) circle (0.08cm); \draw[fill] (4.65,1.6) circle (0.08cm);

\draw[fill] (5.05,0) circle (0.08cm);  \draw[fill] (5.05,0.4) circle (0.08cm);
\draw[fill] (5.05,0.8) circle (0.08cm);
\draw[fill] (5.05,1.2) circle (0.08cm); \draw[fill] (5.05,1.6) circle (0.08cm);

\draw[fill] (5.45,0.4) circle (0.08cm); \draw[fill] (5.45,0.8) circle (0.08cm);
\draw[fill] (5.45,1.2) circle (0.08cm); \draw[fill] (5.45,1.6) circle (0.08cm);

\draw[fill] (5.85,0.8) circle (0.08cm);\draw[fill] (5.85,1.2) circle (0.08cm);
\draw[fill] (5.85,1.6) circle (0.08cm);

\draw[->] (4.45,2.4)--(6.05,2.4); \node at (5.25,2.6) {\footnotesize$-3$};

\draw[->] (6.15,0.8)--(6.55,0.8);
\node at (6.35,2.2) {\footnotesize$1$};\node at (6.35,2.6) {\footnotesize$5$};

\draw[fill] (6.85,0) circle (0.08cm);  \draw[fill] (6.85,0.4) circle (0.08cm);
\draw[fill] (6.85,0.8) circle (0.08cm);
\draw[fill] (6.85,1.2) circle (0.08cm); \draw[fill] (6.85,1.6) circle (0.08cm);

\draw[fill] (7.25,0) circle (0.08cm);  \draw[fill] (7.25,0.4) circle (0.08cm);
\draw[fill] (7.25,0.8) circle (0.08cm);
\draw[fill] (7.25,1.2) circle (0.08cm); \draw[fill] (7.25,1.6) circle (0.08cm);

\draw[fill] (7.65,0.4) circle (0.08cm); \draw[fill] (7.65,0.8) circle (0.08cm);
\draw[fill] (7.65,1.2) circle (0.08cm); \draw[fill] (7.65,1.6) circle (0.08cm);

\draw[fill] (8.05,0.8) circle (0.08cm);\draw[fill] (8.05,1.2) circle (0.08cm);
\draw[fill] (8.05,1.6) circle (0.08cm);

\draw[fill] (8.45,0.8) circle (0.08cm);\draw[fill] (8.45,1.2) circle (0.08cm);
\draw[fill] (8.45,1.6) circle (0.08cm);

\draw[->] (6.65,2.4)--(8.65,2.4); \node at (7.65,2.6) {\footnotesize$-3$};

\draw[->] (8.75,0.8)--(9.15,0.8);
\node at (8.95,2.4) {\footnotesize$3$};

\draw[fill] (9.45,0) circle (0.08cm);  \draw[fill] (9.45,0.4) circle (0.08cm);
\draw[fill] (9.45,0.8) circle (0.08cm);
\draw[fill] (9.45,1.2) circle (0.08cm); \draw[fill] (9.45,1.6) circle (0.08cm);

\draw[fill] (9.85,0) circle (0.08cm);  \draw[fill] (9.85,0.4) circle (0.08cm);
\draw[fill] (9.85,0.8) circle (0.08cm);
\draw[fill] (9.85,1.2) circle (0.08cm); \draw[fill] (9.85,1.6) circle (0.08cm);

\draw[fill] (10.25,0.4) circle (0.08cm); \draw[fill] (10.25,0.8) circle (0.08cm);
\draw[fill] (10.25,1.2) circle (0.08cm); \draw[fill] (10.25,1.6) circle (0.08cm);

\draw[fill] (10.65,0.8) circle (0.08cm);\draw[fill] (10.65,1.2) circle (0.08cm);
\draw[fill] (10.65,1.6) circle (0.08cm);

\draw[fill] (11.05,0.8) circle (0.08cm);\draw[fill] (11.05,1.2) circle (0.08cm);
\draw[fill] (11.05,1.6) circle (0.08cm);

\draw[fill] (11.45,1.2) circle (0.08cm);\draw[fill] (11.45,1.6) circle (0.08cm);

\draw[->] (9.25,2.4)--(11.65,2.4); \node at (10.45,2.6) {\footnotesize$-2$};

\draw[->] (11.75,0.8)--(12.15,0.8);
\node at (11.95,2.4) {\footnotesize$1$};

\draw[fill] (12.45,0) circle (0.08cm);  \draw[fill] (12.45,0.4) circle (0.08cm);
\draw[fill] (12.45,0.8) circle (0.08cm);
\draw[fill] (12.45,1.2) circle (0.08cm); \draw[fill] (12.45,1.6) circle (0.08cm);

\draw[fill] (12.85,0) circle (0.08cm);  \draw[fill] (12.85,0.4) circle (0.08cm);
\draw[fill] (12.85,0.8) circle (0.08cm);
\draw[fill] (12.85,1.2) circle (0.08cm); \draw[fill] (12.85,1.6) circle (0.08cm);

\draw[fill] (13.25,0.4) circle (0.08cm); \draw[fill] (13.25,0.8) circle (0.08cm);
\draw[fill] (13.25,1.2) circle (0.08cm); \draw[fill] (13.25,1.6) circle (0.08cm);

\draw[fill] (13.65,0.8) circle (0.08cm);\draw[fill] (13.65,1.2) circle (0.08cm);
\draw[fill] (13.65,1.6) circle (0.08cm);

\draw[fill] (14.05,0.8) circle (0.08cm);\draw[fill] (14.05,1.2) circle (0.08cm);
\draw[fill] (14.05,1.6) circle (0.08cm);

\draw[fill] (14.45,1.2) circle (0.08cm);\draw[fill] (14.45,1.6) circle (0.08cm);

\draw[fill] (14.85,1.6) circle (0.08cm);

\draw[->] (12.25,2.4)--(14.65,2.4); \node at (13.45,2.6) {\footnotesize$-1$};
\node at (14.95,2.4) {\footnotesize$\emptyset$};

\end{tikzpicture}\caption{An illustration of the bijection $\varphi$.}\label{figbije}
\end{center}
\end{figure}
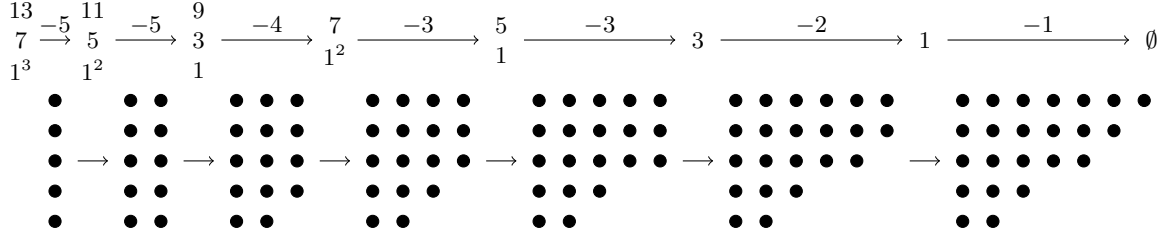

We now know that the gap-free sequence $\mathbf{s}$ plays a crucial role in the bijection $\varphi$. On the one hand, it records the quantity reduced each time when we perform the operation $\tau$. In fact, this quantity equals the block odd index of one odd partition. On the other hand, if we view $\mathbf{s}$ as a partition, then it is the conjugate of one distinct partition. Namely, based on $\mathbf{s}$, we can produce an odd partition $\lambda$ and a distinct partition $\pi$ such that $\pi=\varphi(\lambda)$. From now on, we say that $\mathbf{s}$ is the gap-free sequence associated with $\lambda$ or corresponding to $\lambda$.

\section{Statistics}\label{secsta}

Given an odd partition $\lambda$, we always use $\pi$ to denote $\varphi(\lambda)$ in the sequel. In this section, we aim to prove $\ell(\lambda)=\ell_a(\pi),o_{1,4}(\lambda)=n_{o,o}(\pi),o_{3,4}(\lambda)=n_{o,e}(\pi)$ and
$b_o(\lambda)=\ell(\pi)$.

For any odd partition $\lambda$, we know from the definition of $\tau$ that $\ell(\tau(\lambda))=\ell(\lambda)-1$ if and only $b_o(\lambda)$ is odd. Therefore, the length of $\lambda$ is equal to the number of times that $b_o(\tau^{i}(\lambda))$ is odd when we perform the operation $\tau$ until the partition is completely emptied, which equals the number of odd integers in the sequence $\mathbf{s}$. Recall that $\mathbf{s}$ is essentially the conjugate of $\pi$. Moreover, $\pi_{2i-1}-\pi_{2i}$ is the number of parts of size $2i-1$ in the conjugate of $\pi$. Thus, \[\ell_a(\pi)=(\pi_{1}-\pi_{2})+(\pi_3-\pi_4)+\cdots\] counts the number of odd integers in $\mathbf{s}$. We now conclude that $\ell(\lambda)=\ell_a(\pi)$.

To prove $o_{1,4}(\lambda)=n_{o,o}(\pi)$ and $o_{3,4}(\lambda)=n_{o,e}(\pi)$, we need some additional statistics, which are explored from the bijection $\varphi$.

To perform $\varphi$, we need to group the parts of $\lambda$ into blocks from largest to smallest. If a block of weight $2$ has only one distinct integer, letting it be $(t^a)$ where $a\geq 1$, we change it to be $(t^a,(t-2)^0)$. Namely, we regard $t-2$ as a part of $\lambda$, which just occurs zero times. So every block of weight $2$ has now two distinct integers. We still use parentheses $(\,)$ to denote this type of blocks. For example, if $\lambda=(21^2,19^3,17^5,13^2,7^8,5^2,1^6)$, then
\[\lambda=(21^2,19^3)(17^5,15^0)(13^2,11^0)(7^8,5^2)(1^6).\]
Let $e_{1,4}(\lambda)$ and  $e_{3,4}(\lambda)$ be the number of distinct parts in $\lambda$ which occur an even number of times (including zero times) and are congruent to $1$ and $3$ modulo $4$, respectively. Continuing the above $\lambda$, we have $e_{1,4}(\lambda)=4$ (counting $21,13,5,1$), and $e_{3,4}(\lambda)=3$ (counting $15,11,7$). It is clear that
\[o_{1,4}(\lambda)+o_{3,4}(\lambda)+e_{1,4}(\lambda)+e_{3,4}(\lambda)=b_o(\lambda).\]

To perform $\varphi^{-1}$, we need to group the parts from smallest to largest and take into account the parity of the weight added. We know from the definition of $\lambda\stackrel{+w}{\longrightarrow}$ that we will get a block of weight $1$ only when $w$ is odd. Thus, we think of each block as having two distinct integers for all other cases. We use square brackets $[\,\,]$ to denote this new type of blocks. For instance, take $\lambda=(21^3,15^5,9^4,5^2,3^4,1^6)$, where $b_o(\lambda)=9$. If $w=9$, then
\[\lambda=[23^0,21^3][17^0,15^5][11^0,9^4][5^2,3^4][1^6].\]
If $w=10$, then
\[\lambda=[23^0,21^3][17^0,15^5][11^0,9^4][7^0,5^2][3^4,1^6].\]
When we perform $\lambda\stackrel{+w}{\longrightarrow}$, let $w_{1,4}(\lambda)$ and $w_{3,4}(\lambda)$ be the number of distinct parts in $\lambda$ which occur an even number of times (including zero times) and are congruent to $1$ and $3$ modulo $4$, respectively. Continue the above example. If $w=9$, then $w_{1,4}(\lambda)=4$ (counting $17,9,5,1$) and $w_{3,4}(\lambda)=3$ (counting $23,11,3$). If $w=10$, then $w_{1,4}(\lambda)=4$ (counting $17,9,5,1$) and $w_{3,4}(\lambda)=4$ (counting $23,11,7,3$).

For a distinct partition $\pi$, let $n_{e,o}(\pi)$ and $n_{e,e}(\pi)$ be the number of odd and even indexed even parts of $\pi$, respectively.

We next show that
\begin{align}\label{eqon}
o_{1,4}(\lambda)=n_{o,o}(\pi),
o_{3,4}(\lambda)=n_{o,e}(\pi),
e_{1,4}(\lambda)=n_{e,o}(\pi),
e_{3,4}(\lambda)=n_{e,e}(\pi).
\end{align}

Recall that both $\lambda$ and $\pi$ are determined uniquely by the gap-free sequence $\mathbf{s}$. In essence, we construct $\lambda$ iteratively by performing the weight-adding operation whereas we obtain the Ferrers graph of $\pi$ by adding one column each time from right to left. To perform $\varphi^{-1}$, it is convenient to write $\mathbf{s}$ in the weakly increasing order of its entries. For
$k\geq1$, let $\mathbf{s}_k$ denote the prefix of $\mathbf{s}$ of length $k$, i.e., $\mathbf{s}_k=(s_1,s_2,\ldots,s_k)$ if $\mathbf{s}=(s_1,s_2,\ldots,s_n)$. Clearly, $\mathbf{s}_k$ is also a gap-free sequence, and determines an odd partition $\Lambda_k$ and a distinct partition $\Pi_k$ such that $b_o(\Lambda_k)=\ell(\Pi_k)=s_k$. We now proceed by induction on $k$ to prove that
\begin{align}\label{eqinduc}
o_{1,4}(\Lambda_k)=n_{o,o}(\Pi_k), o_{3,4}(\Lambda_k)=n_{o,e}(\Pi_k), e_{1,4}(\Lambda_k)=n_{e,o}(\Pi_k), e_{3,4}(\Lambda_k)=n_{e,e}(\Pi_k).
\end{align}
The case for $k=1$ is trivial since $s_1=1$ and $\Lambda_1=(1)=\Pi_1$. Assume that the assertion is true for $k-1$. Performing the operation $\Lambda_{k-1}\stackrel{+s_k}{\longrightarrow}$, we get $\Lambda_k$. Adding a new column of length $s_k$ to the Ferrers graph of $\Pi_{k-1}$, we obtain the Ferrers graph of $\Pi_k$.

Each block produced by parentheses-grouping the parts of $\Lambda_k$ has the same integers as that block produced by square-brackets-grouping the parts of $\Lambda_{k-1}$ when we add the weight $s_k$. Thus, we always have
\begin{align}
&\hbox{$o_{1,4}(\Lambda_k)=w_{1,4}(\Lambda_{k-1})$ and $o_{3,4}(\Lambda_k)=w_{3,4}(\Lambda_{k-1})$},\label{eqow}\\
&\hbox{$e_{1,4}(\Lambda_k)=o_{1,4}(\Lambda_{k-1})$ and $e_{3,4}(\Lambda_k)=o_{3,4}(\Lambda_{k-1})$}.\label{eqeo}
\end{align}

For an odd partition $\lambda$, a chain of $\lambda$ is a maximal sequence of parts consisting of consecutive odd integers. The number of distinct integers appearing in a chain is called the length of the chain. For example, there are $4$ chains in $(25^3,23^2,21^5,17^7,15^4,11^9,7^2,5,3^8,1^6)$, which are $(25^3,23^2,21^5)$, $(17^7,15^4)$, $(11^9)$ and $(7^2,5,3^8,1^6)$, respectively. For a chain of even length, we group its elements pairwise into blocks from largest to smallest is the same as we group from smallest to largest. For instance,
\[(13,11^3,9^2,7^5,5^4,3^6)=(13,11^3)(9^2,7^5)(5^4,3^6)=[13,11^3][9^2,7^5][5^4,3^6].\]
Therefore, we have $e_{1,4}(\lambda)=w_{1,4}(\lambda)$ and $e_{3,4}(\lambda)=w_{3,4}(\lambda)$ where $\lambda$
is a chain of even length, and also regarded as an odd partition. However, the two ways of grouping parts for a chain
of odd length are different. For example, if $\lambda=(17^5,15^2,13,11^3,9^8,7^4,5^6)$, we get \[(17^5,15^2)(13,11^3)(9^8,7^4)(5^6,3^0)\]
if we group from largest to smallest while we get
\[[19^0,17^5][15^2,13][11^3,9^8][7^4,5^6]\]
if we group from smallest to largest. But we still have $e_{1,4}(\lambda)=w_{1,4}(\lambda)$ and $e_{3,4}(\lambda)=w_{3,4}(\lambda)$ for a chain of odd length since the last integer (added when grouping from largest to smallest) is congruent to the first integer (added when grouping from smallest to largest) modulo $4$.

Recall that $s_k=b_o(\Lambda_{k-1})$ or $s_k=b_o(\Lambda_{k-1})+1$. So we have two cases to discuss.

If $s_k=b_o(\Lambda_{k-1})$, then we have $e_{1,4}(\Gamma)=w_{1,4}(\Gamma)$ and $e_{3,4}(\Gamma)=w_{3,4}(\Gamma)$ for each chain $\Gamma$ of $\Lambda_{k-1}$. Therefore, $e_{1,4}(\Lambda_{k-1})=w_{1,4}(\Lambda_{k-1})$ and $e_{3,4}(\Lambda_{k-1})=w_{3,4}(\Lambda_{k-1})$. Combining with \eqref{eqow}, we obtain
$o_{1,4}(\Lambda_k)=e_{1,4}(\Lambda_{k-1})$ and $o_{3,4}(\Lambda_k)=e_{3,4}(\Lambda_{k-1})$. We now arrive at
\begin{align*}
\hbox{$o_{1,4}(\Lambda_k)=e_{1,4}(\Lambda_{k-1})$ and $o_{3,4}(\Lambda_k)=e_{3,4}(\Lambda_{k-1})$},\\
\hbox{$e_{1,4}(\Lambda_k)=o_{1,4}(\Lambda_{k-1})$ and $e_{3,4}(\Lambda_k)=o_{3,4}(\Lambda_{k-1})$}.
\end{align*}
On the other hand, since $s_k=b_o(\Lambda_{k-1})=\ell(\Pi_{k-1})$ and $\ell(\Pi_k)=s_k$, implying $\ell(\Pi_k)=\ell(\Pi_{k-1})$, we obtain that
\begin{align*}
\hbox{$n_{o,o}(\Pi_k)=n_{e,o}(\Pi_{k-1})$ and $n_{o,e}(\Pi_k)=n_{e,e}(\Pi_{k-1})$},\\
\hbox{$n_{e,o}(\Pi_k)=n_{o,o}(\Pi_{k-1})$ and $n_{e,e}(\Pi_k)=n_{o,e}(\Pi_{k-1})$}.
\end{align*}
By the induction hypothesis, we know that
\begin{align*}
\hbox{$o_{1,4}(\Lambda_{k-1})=n_{o,o}(\Pi_{k-1})$ and $o_{3,4}(\Lambda_{k-1})=n_{o,e}(\Pi_{k-1})$},\\
\hbox{$e_{1,4}(\Lambda_{k-1})=n_{e,o}(\Pi_{k-1})$ and $e_{3,4}(\Lambda_{k-1})=n_{e,e}(\Pi_{k-1})$}.
\end{align*}
Combining the above results together, we prove \eqref{eqinduc} for the case $s_k=b_o(\Lambda_{k-1})$.

If $s_k=b_o(\Lambda_{k-1})+1$, then $e_{1,4}(\Gamma)=w_{1,4}(\Gamma)$ and $e_{3,4}(\Gamma)=w_{3,4}(\Gamma)$ for each chain $\Gamma$ of $\Lambda_{k-1}$ except the chain containing the parts of size $1$. We now assume that $\Gamma$ is the chain where $1$'s lie in. There are two subcases to be considered according to the parity of $s_k$. If $s_k$ is even, then $b_o(\Lambda_{k-1})$ is odd, and the length of $\Gamma$ is odd. Thus, $w_{1,4}(\Gamma)=e_{1,4}(\Gamma)$ and $w_{3,4}(\Gamma)=e_{3,4}(\Gamma)+1$. For example, let $\Gamma=(9^4,7,5^3,3^8,1^6)$. To compute $e_{1,4}(\Gamma)$ and $e_{3,4}(\Gamma)$, $\Gamma$ is grouped as
$(9^4,7)(5^3,3^8)(1^6)$. So $e_{1,4}(\Gamma)=2$ and $e_{3,4}(\Gamma)=1$. To compute $w_{1,4}(\Gamma)$ and $w_{3,4}(\Gamma)$, we group $\Gamma$ as $[11^0,9^4][7,5^3][3^8,1^6]$, and obtain $w_{1,4}(\Gamma)=2$ and $w_{3,4}(\Gamma)=2$. We now conclude that $w_{1,4}(\Lambda_{k-1})=e_{1,4}(\Lambda_{k-1})$ and $w_{3,4}(\Lambda_{k-1})=e_{3,4}(\Lambda_{k-1})+1$. If $s_k$ is odd, then $b_o(\Lambda_{k-1})$ is even, and the length of $\Gamma$ is even if $1$ occurs in $\Lambda_{k-1}$; otherwise $\Gamma$ is empty. Since the weight $s_k$ is odd, all $1$'s will inhabit a block by themselves. Hence, if $\Gamma$ is nonempty, then $w_{1,4}(\Gamma)=e_{1,4}(\Gamma)+1$ and $w_{3,4}(\Gamma)=e_{3,4}(\Gamma)$; if $\Gamma$ is empty, then
$w_{1,4}(\Gamma)=1$, $e_{1,4}(\Gamma)=0$, and $w_{3,4}(\Gamma)=e_{3,4}(\Gamma)=0$, still satisfying
$w_{1,4}(\Gamma)=e_{1,4}(\Gamma)+1$ and $w_{3,4}(\Gamma)=e_{3,4}(\Gamma)$. We now conclude that $w_{1,4}(\Lambda_{k-1})=e_{1,4}(\Lambda_{k-1})+1$ and $w_{3,4}(\Lambda_{k-1})=e_{3,4}(\Lambda_{k-1})$. Combining with \eqref{eqow}, we obtain
\begin{align*}
o_{1,4}(\Lambda_k)&=w_{1,4}(\Lambda_{k-1})=
\left\{
\begin{array}{ll}
e_{1,4}(\Lambda_{k-1}),   & \hbox{if $s_k$ is even;} \\
e_{1,4}(\Lambda_{k-1})+1, & \hbox{if $s_k$ is odd,}
\end{array}
\right.\\
o_{3,4}(\Lambda_k)&=w_{3,4}(\Lambda_{k-1})=
\left\{
\begin{array}{ll}
e_{3,4}(\Lambda_{k-1})+1,   & \hbox{if $s_k$ is even;} \\
e_{3,4}(\Lambda_{k-1}), & \hbox{if $s_k$ is odd.}
\end{array}
\right.
\end{align*}
On the other hand, since $s_k=\ell(\Pi_{k-1})+1$, implying $\ell(\Pi_k)=\ell(\Pi_{k-1})+1$, we obtain that
\begin{align*}
n_{o,o}(\Pi_k)&=
\left\{
\begin{array}{ll}
n_{e,o}(\Pi_{k-1}),   & \hbox{if $s_k$ is even;} \\
n_{e,o}(\Pi_{k-1})+1, & \hbox{if $s_k$ is odd,}
\end{array}
\right.\\
n_{o,e}(\Pi_k)&=
\left\{
\begin{array}{ll}
n_{e,e}(\Pi_{k-1})+1, & \hbox{if $s_k$ is even;} \\
n_{e,e}(\Pi_{k-1}),   & \hbox{if $s_k$ is odd,}
\end{array}
\right.\\
n_{e,o}(\Pi_k)&=n_{o,o}(\Pi_{k-1}),\\
n_{e,e}(\Pi_k)&=n_{o,e}(\Pi_{k-1}).
\end{align*}
Combining all the above results with the induction hypothesis and \eqref{eqeo} together, we finish the proof of \eqref{eqinduc} for the case $s_k=b_o(\Lambda_{k-1})+1$.

Because $\lambda=\Lambda_{n}$ and $\pi=\Pi_n$, we conclude that \eqref{eqon} is true.

The above induction process also shows that $b_o(\Lambda_k)=\ell(\Pi_k)$ for each $k$. Therefore, we have $b_o(\lambda)=\ell(\pi)$. In addition, it follows immediately from the visual description of $\varphi$ that $b_o(\lambda)=\ell(\pi)$.

\section{Lecture hall length}\label{seclh}

We now clearly know that if $\lambda\in\mathcal{O}$ and $\pi=\varphi(\lambda)$, then $\lambda$ and $\pi$ are essentially determined by the gap-free sequence $\mathbf{s}$. Furthermore, if we regard $\mathbf{s}$ as a partition, then $\mathbf{s}$ is the conjugate of $\pi$. For convenience, we write $\mathbf{s}$ in the compact form
\[\mathbf{s}=(1^{m_1},2^{m_2},\ldots,k^{m_k}),\]
where each term $m_i$ is positive and denotes the number of repetitions of $i$ in $\mathbf{s}$. Suppose that $\pi=(\pi_1,\pi_2,\ldots,\pi_k)$. Then $m_1=\pi_1-\pi_2$, $m_2=\pi_2-\pi_3,\ldots,m_{k-1}=\pi_{k-1}-\pi_k$ and $m_k=\pi_k$. Conversely, if $\mathbf{s}$ is given, then
\begin{align*}
\pi_1&=m_1+m_2+\cdots+m_k,\\
\pi_2&=m_2+m_3+\cdots+m_k,\\
&\,\,\,\vdots\\
\pi_k&=m_k.
\end{align*}

In the remainder of this section, we will discuss the properties of a distinct partition as a lecture hall partition. For simplicity, we always assume that $\pi$ is a distinct partition of length $k$, i.e., $\pi=(\pi_1,\pi_2,\ldots,\pi_k)$, and $\pi_i=0$ if $i>k$. We use $\mathbf{s}=(1^{m_1},2^{m_2},\ldots,k^{m_k})$ to denote the conjugate of $\pi$. So we have $m_i=\pi_i-\pi_{i+1}$ for $1\leq i\leq k$. We will frequently use these notation without mentioning them explicitly.

\begin{lemma}\label{lemlh}
If
\begin{align*}
t\geq\max\limits_{1\leq i\leq k}\left\{\frac{\pi_i}{\pi_i-\pi_{i+1}}+i-1\right\},
\end{align*}
then for $1\leq i\leq k$, we have
\begin{align*}
\frac{\pi_{i}}{t-i+1}\geq\frac{\pi_{i+1}}{t-i}.
\end{align*}
\end{lemma}
\begin{proof}
The condition
\begin{align*}
t\geq\frac{\pi_i}{\pi_i-\pi_{i+1}}+i-1
\end{align*}
is equivalent to
\begin{align*}
t\pi_i-t\pi_{i+1}&\geq\pi_i+(i-1)\pi_i-(i-1)\pi_{i+1}\\
&=i\pi_i-(i-1)\pi_{i+1},
\end{align*}
which can be rewritten as $(t-i)\pi_i\geq(t-i+1)\pi_{i+1}$ or
\[\frac{\pi_{i}}{t-i+1}\geq\frac{\pi_{i+1}}{t-i}.\]
This completes the proof.
\end{proof}

From the proof of Lemma \ref{lemlh}, we can see that
\begin{align*}
t\geq\max\limits_{1\leq i\leq k}\left\{\frac{\pi_i}{\pi_i-\pi_{i+1}}+i-1\right\}
\end{align*}
is equivalent to
\[\frac{\pi_{i}}{t-i+1}\geq\frac{\pi_{i+1}}{t-i},\,\,\mbox{for $1\leq i\leq k$}.\]
This enables us to conclude a formula to compute the lecture hall length of $\pi$. Let $\lceil x\rceil$ denote the least integer greater than or equal to $x$.
\begin{theorem}\label{thmellh}
We have
\begin{align*}
\ell_h(\pi)=\max_{1\leq i\leq k}\left\{\left\lceil\frac{\pi_i}{\pi_i-\pi_{i+1}}\right\rceil+i-1\right\}.
\end{align*}
\end{theorem}
\begin{example}
Letting $\pi=(7,5,4,1)$, then
\[\left\lceil\frac{7}{7-5}\right\rceil=4,\left\lceil\frac{5}{5-4}\right\rceil+1=6,\left\lceil\frac{4}{4-1}\right\rceil+2=4,
\left\lceil\frac{1}{1-0}\right\rceil+3=4.\]
Thus, $\ell_h(\pi)=6$. One can check that
\begin{align*}
&\frac{7}{4}\geq\frac{5}{3}\not\geq\frac{4}{2}\geq\frac{1}{1},\\[3pt]
&\frac{7}{5}\geq\frac{5}{4}\not\geq\frac{4}{3}\geq\frac{1}{2}\geq\frac{0}{1},\\[3pt]
&\frac{7}{6}\geq\frac{5}{5}\geq\frac{4}{4}\geq\frac{1}{3}\geq\frac{0}{2}\geq\frac{0}{1}.
\end{align*}
\end{example}

\begin{theorem}\label{thms}
Assume that $\pi=(\pi_1,\pi_2,\ldots,\pi_k)\in\mathcal{D}$ and $\pi'=(1^{m_1},2^{m_2},\ldots,k^{m_k})$. If
\begin{align*}
\frac{\pi_j}{m_j}+j-1=\max\limits_{1\leq i\leq k}\left\{\frac{\pi_i}{m_i}+i-1\right\},
\end{align*}
then $\pi_i\geq\pi_j+(j-i)m_j$ for $i\leq j$, and $\pi_j+(j-i)m_j\geq\pi_i$ for $i>j$.
\end{theorem}
\begin{proof}
For $i\leq j$, it follows from Lemma \ref{lemlh} (setting $t=\pi_j/m_j+j-1$) that
\begin{align*}
\frac{\pi_i}{(\pi_j/m_j+j-1)-i+1}\geq\frac{\pi_j}{(\pi_j/m_j+j-1)-j+1}.
\end{align*}
After simplification, we obtain
\begin{align*}
\frac{\pi_i}{\pi_j/m_j+j-i}\geq m_j,
\end{align*}
which is equivalent to $\pi_i\geq\pi_j+(j-i)m_j$.

The proof for the case $i>j$ is similar and is omitted here.
\end{proof}

\begin{remark}
If we change the condition in Theorem \ref{thms} to
\begin{align*}
\left\lceil\frac{\pi_j}{m_j}\right\rceil+j-1=\max\limits_{1\leq i\leq k}\left\{\left\lceil\frac{\pi_i}{m_i}\right\rceil+i-1\right\},
\end{align*}
then the result fails to hold. For example, if $\pi=(40,31,21,12,3)$, we have
\begin{align*}
\left\lceil\frac{40}{9}\right\rceil=\left\lceil\frac{31}{10}\right\rceil+1
=\left\lceil\frac{21}{9}\right\rceil+2=\left\lceil\frac{12}{9}\right\rceil+3
=\left\lceil\frac{3}{3}\right\rceil+4=5.
\end{align*}
Taking $\pi_j=31$, we find that $40\not\geq31+10$, $31-20\not\geq12$ and $31-30\not\geq3$.
\end{remark}

Suppose that $\pi=(\pi_1,\pi_2,\ldots,\pi_k)\in\mathcal{D}$ and $\pi'=(1^{m_1},2^{m_2},\ldots,k^{m_k})$. We say that $j$ is the lecture hall position if it satisfies the following three conditions:
\begin{itemize}
\item[(a)] $\left\lceil\pi_j/m_j\right\rceil+j-1=\ell_h(\pi)$;
\item[(b)] Assume that $\left\lceil\pi_i/m_i\right\rceil+i-1=\left\lceil\pi_j/m_j\right\rceil+j-1=\ell_h(\pi)$.
           Letting $r_i=\pi_i-(\ell_h(\pi)-i)m_i$ and $r_j=\pi_j-(\ell_h(\pi)-j)m_j$, then $r_j\geq r_i$;
\item[(c)] $j$ is the smallest integer satisfying both conditions (a) and (b).
\end{itemize}
\begin{remark}
Note that if $\left\lceil\pi_j/m_j\right\rceil+j-1=\ell_h(\pi)$, then there is a unique way to write \[\pi_j=(\ell_h(\pi)-j)m_j+r_j\]
where $1\leq r_j\leq m_j$ since $\ell_h(\pi)-j<\pi_j/m_j\leq \ell_h(\pi)-j+1$.
\end{remark}

Sometimes we will say that the part $\pi_j$  or the weight $j$ in $\mathbf{s}$ ($=\pi'$) determines the lecture hall position. We will use the above statements interchangeably.

\begin{example}
Letting $\pi=(74,65,40,34,24,19)$, then $\mathbf{s}=\pi'=(1^{9},2^{25},3^6,4^{10},5^5,6^{19})$. We first compute the lecture hall length as follows
\begin{align*}
\left\lceil{\pi_1/m_1}\right\rceil+0&=\left\lceil{74/9}\right\rceil+0=9,\\
\left\lceil{\pi_2/m_2}\right\rceil+1&=\left\lceil{65/25}\right\rceil+1=4,\\
\left\lceil{\pi_3/m_3}\right\rceil+2&=\left\lceil{40/6}\right\rceil+2=9,\\
\left\lceil{\pi_4/m_4}\right\rceil+3&=\left\lceil{34/10}\right\rceil+3=7,\\
\left\lceil{\pi_5/m_5}\right\rceil+4&=\left\lceil{24/5}\right\rceil+4=9,\\
\left\lceil{\pi_6/m_6}\right\rceil+5&=\left\lceil{19/19}\right\rceil+5=6.
\end{align*}
We now see that $\ell_h(\pi)=9$, so the lecture hall position is in the set $\{1,3,5\}$. Because
\begin{align*}
\pi_1&=74=8\cdot9+2,\\
\pi_3&=40=6\cdot6+4,\\
\pi_5&=24=4\cdot5+4,
\end{align*}
we conclude that the part $\pi_3$ determines the lecture hall position.
\end{example}

\begin{theorem}
Assume that $\pi=(\pi_1,\pi_2,\ldots,\pi_k)\in\mathcal{D}$ and $\pi'=(1^{m_1},2^{m_2},\ldots,k^{m_k})$. Then the lecture hall position cannot be $k$.
\end{theorem}
\begin{proof}
If $k$ is the lecture hall position, then $\ell_h(\pi)=k$. Thus,
\[\frac{\pi_1}{k}\geq\frac{\pi_{2}}{k-1}\geq\cdots\geq\frac{\pi_{k-1}}{2}\geq\frac{\pi_k}{1},\]
which implies that $\pi_{k-1}\geq2\pi_k$. Moreover, we have $m_{k-1}\geq m_k$ since $\pi_{k-1}=m_{k-1}+m_k$ and $\pi_k=m_k$.

On the other hand,
\begin{align*}
\left\lceil\frac{\pi_{k-1}}{m_{k-1}}\right\rceil+k-2=\left\lceil\frac{m_{k-1}+m_k}{m_{k-1}}\right\rceil+k-2=k.
\end{align*}
We now have
\begin{align*}
\left\lceil\pi_{k-1}/m_{k-1}\right\rceil+k-2&=\ell_h(\pi),\\
\left\lceil\pi_k/m_k\right\rceil+k-1&=\ell_h(\pi),\\
\pi_{k-1}&=1\cdot m_{k-1}+m_k,\\
\pi_k&=0\cdot m_k+m_k.
\end{align*}
This means that the lecture hall position should be $k-1$ rather than $k$, a contradiction.
\end{proof}

\begin{theorem}\label{thmpm}
Given $\pi=(\pi_1,\pi_2,\ldots,\pi_k)\in\mathcal{D}$ and $\pi'=(1^{m_1},2^{m_2},\ldots,k^{m_k})$, if the weight $j$ determines the lecture hall position, then $m_j<\min\{m_1,m_2,\ldots,m_{j-1}\}$. Moreover, we have $m_j\leq m_{j+1}$ provided that $\ell_h(\pi)\neq k$.
\end{theorem}
\begin{proof}
First, we assume that $m_i=\min\{m_1,m_2,\ldots,m_{j-1}\}$ and
\begin{align*}
\pi_i&=(\ell_h(\pi)-i)m_i+r_i,\\
\pi_j&=(\ell_h(\pi)-j)m_j+r_j,\\
\pi_{j+1}&=(\ell_h(\pi)-j-1)m_{j+1}+r_{j+1}.
\end{align*}

If $m_j\geq m_i$, then it follows from Theorem \ref{thms} that
\begin{align*}
\pi_i&\geq\pi_j+(j-i)m_j\\
&=(\ell_h(\pi)-j)m_j+r_j+(j-i)m_j\\
&=(\ell_h(\pi)-i)m_j+r_j\\
&\geq(\ell_h(\pi)-i)m_i+r_j,
\end{align*}
which implies that $r_i\geq r_j$ and
\begin{align*}
{\pi_i}/{m_i}+i-1\geq(\ell_h(\pi)-i)+r_j/m_i+i-1=\ell_h(\pi)-1+r_j/m_i.
\end{align*}
So $\left\lceil{\pi_i}/{m_i}\right\rceil+i-1\geq\ell_h({\pi})$, hence $\left\lceil{\pi_i}/{m_i}\right\rceil+i-1=\ell_h({\pi})$. We now have
\[\left\lceil{\pi_i}/{m_i}\right\rceil+i-1=\left\lceil{\pi_j}/{m_j}\right\rceil+j-1=\ell_h({\pi})\]
and $r_i\geq r_j$, which implies that $\pi_j$ cannot be the lecture hall position, a contradiction.

When $\ell_h(\pi)\neq k$, we have $\ell_{h}(\pi)>k\geq j+1$, implying $\ell_h(\pi)-j-1>0$. If $m_j>m_{j+1}$, then we obtain
\begin{align*}
\pi_{j+1}&=\pi_j-m_{j}\\
&=(\ell_h(\pi)-j)m_j+r_j-m_j\\
&=(\ell_h(\pi)-j-1)m_j+r_j\\
&>(\ell_h(\pi)-j-1)m_{j+1}+r_j,
\end{align*}
which implies that $r_{j+1}>r_j$ and
\begin{align*}
{\pi_{j+1}}/{m_{j+1}}+j>(\ell_h(\pi)-j-1)+r_j/m_{j+1}+j=\ell_h(\pi)-1+r_j/m_{j+1}.
\end{align*}
Thus, $\left\lceil{\pi_{j+1}}/{m_{j+1}}\right\rceil+j\geq\ell_h(\pi)$, so $\left\lceil{\pi_{j+1}}/{m_{j+1}}\right\rceil+j=\ell_h(\pi)$.
Therefore, $j$ cannot be the lecture hall position since $r_{j+1}>r_j$, a contradiction.
\end{proof}
\begin{remark}\label{remmj}
When $\ell_h(\pi)=k$, there are two cases to consider. If $j<k-1$, we still have $\ell_h(\pi)-j-1>0$. Using the same arguments, we can derive that $m_j\leq m_{j+1}$. If $j=k-1$, then we must have $m_j\geq m_{j+1}$ otherwise $\ell_h(\pi)>k$. Combining with Theorem \ref{thmpm}, we can say that $m_j\leq m_{j+1}$ always holds whenever $j<k-1$.
\end{remark}

\section{Chains of odd partitions}\label{secchain}

Recall that a chain of an odd partition is a maximal sequence of parts consisting of consecutive odd integers, and the length of a chain is the number of distinct integers appearing in it. For a chain $\Gamma$, we still use $\ell(\Gamma)$ to denote its length. A chain can be clearly viewed as an odd partition. Assume that $\Gamma_1,\Gamma_2,\ldots,\Gamma_k$ (arranged in the decreasing order of their largest parts) are all chains of $\lambda$. Then the chain decomposition of $\lambda$ is
\[\lambda=\Gamma_1\cup\Gamma_2\cup\cdots\cup\Gamma_k,\]
where the union $\alpha\cup\beta$ denotes the partition consisting of all parts in $\alpha$ and $\beta$. For $i\geq1$, the smallest part of $\Gamma_i$ is at least $4$ plus the largest part of $\Gamma_{i+1}$. Therefore, we see that
\begin{align*}
b_o(\lambda)&=b_o(\Gamma_1)+b_o(\Gamma_2)+\cdots+b_o(\Gamma_k),\\
\tau(\lambda)&=\tau(\Gamma_1)\cup\tau(\Gamma_2)\cup\cdots\cup\tau(\Gamma_k).
\end{align*}
Note that $\tau(\Gamma_i)$ may not be a chain of $\tau(\lambda)$ any more since $\tau(\Gamma_i)$ may become the union of several shorter chains of $\tau(\lambda)$, or $\tau(\Gamma_i)$ form a chain of $\tau(\lambda)$ with other $\tau(\Gamma_{j})$ together.
\begin{example}
Letting $\lambda=(27^2,25^3,21^4,19^4,17,15^8,9,7^5,5,3^6,1^8)$, then it has three chains:
$\Gamma_1=(27^2,25^3)$, $\Gamma_2=(21^4,19^4,17,15^8)$ and $\Gamma_3=(9,7^5,5,3^6,1^8)$. And we have
$b_o(\lambda)=11$, $b_o(\Gamma_1)=2$, $b_o(\Gamma_2)=4$ and $b_o(\Gamma_3)=5$ satisfying $b_o(\lambda)=b_o(\Gamma_1)+b_o(\Gamma_2)+b_o(\Gamma_3)$. In addition, $\tau(\lambda)=(27,25^4,21^3,19^5,15^9,7^6,3^7,1^7)$ and
\begin{align*}
\tau(\Gamma_1)&=(27,25^4),\\
\tau(\Gamma_2)&=(21^3,19^5,15^9)=(21^3,19^5)\cup(15^9),\\
\tau(\Gamma_3)&=(7^6,3^7,1^7)=(7^6)\cup(3^7,1^7).
\end{align*}
We now see that $\tau(\lambda)$ has five chains.
\end{example}

Given a chain $\Gamma$, we group its parts into blocks, and define the multiplicity of $\Gamma$ to be the minimum value of the multiplicity of the maximal part of every block in $\Gamma$, denoted $\varepsilon(\Gamma)$ or just $\varepsilon$ when there is no possibility of confusion. For $1\leq i<\varepsilon$, $\tau^{i}(\Gamma)$ is still a chain. However, $\tau^{\varepsilon}(\Gamma)$ will be disconnected into several shorter chains. For instance, if \[\Gamma=(19^8,17^5,15^6,13^2,11^4,9,7^5,5^3,3^7)=(19^8,17^5)(15^6,13^2)(11^4,9)(7^5,5^3)(3^7),\]
then $\varepsilon=4=\min\{8,6,4,5,7\}$, and
\begin{align*}
\tau(\Gamma)&=(19^7,17^6)(15^5,13^3)(11^3,9^2)(7^4,5^4)(3^6,1),\\
\tau^2(\Gamma)&=(19^6,17^7)(15^4,13^4)(11^2,9^3)(7^3,5^5)(3^5,1^2),\\
\tau^3(\Gamma)&=(19^5,17^8)(15^3,13^5)(11,9^4)(7^2,5^6)(3^4,1^3),\\
\tau^4(\Gamma)&=(19^4,17^9)(15^2,13^6)\cup(9^5,7)(5^7,3^3)(1^4)\\
&=(19^4,17^9,15^2,13^6)\cup(9^5,7,5^7,3^3,1^4).
\end{align*}
For convenience, we often from now on represent a chain as the block form. Let $t$ be the first part of $\Gamma$ (reading from left to right) which is the maximal part of a block and occurs exactly $\varepsilon$ times in $\Gamma$, i.e.,
\[\Gamma=\cdots(t+4^{>\varepsilon},t+2^{>0})(t^{=\varepsilon},t-2^{>0})(t-4^{\geq\varepsilon},t-6^{>0})\cdots,\]
where the notation $t^{>r}$, $t^{\geq r}$ and $t^{=r}$ means that $t$ as a part of $\Gamma$ appears more than $r$ times, at least $r$ times and exactly $r$ times, respectively. Denote by $\Gamma_\varepsilon$ the subchain of $\Gamma$ consisting of the parts greater than $t$. Note that $\Gamma_\varepsilon$ may be empty. If $\Gamma_\varepsilon$ is nonempty, $\Gamma_\varepsilon$ will be separated from $\Gamma$ after performing the operation $\tau$ exactly $\varepsilon$ times. Furthermore, $\tau^{\varepsilon}(\Gamma_\varepsilon)$ is the first chain of $\tau^{\varepsilon}(\Gamma)$, and has even length. Continuing the previous example, we have $\Gamma_\varepsilon=(19^8,17^5)(15^6,13^2)$, and $\tau^4(\Gamma_\varepsilon)=(19^4,17^9)(15^2,13^6)$, which is the first chain of $\tau^4(\Gamma)$ and has even length. In fact, if $\Gamma$ is a chain without $1$'s, then each chain of $\tau^{\varepsilon}(\Gamma)$ except $\tau^{\varepsilon}(\Gamma_\varepsilon)$ has odd length.

A chain $\Gamma$ without $1$'s is called $\delta$-fixed if the multiplicity of each chain of odd length in $\tau^{\varepsilon}(\Gamma)$ is at least $\delta$, and $\Gamma_{\varepsilon}$ is also a $\delta$-fixed chain when $\Gamma_{\varepsilon}$ is nonempty.
\begin{example}\label{exafix}
We show that $\Gamma$ is a $6$-fixed chain, where
\begin{align*}
\Gamma=(37^7,35^2)(33^5,31)(29^8,27^5)(25^{11},23^4)(21^3,19^{12})(17^4,15^6)(13^4,11^5)(9^3,7^{10})(5^6,3^7).
\end{align*}
Firstly, we have $\varepsilon=3$ and $\Gamma_\varepsilon=(37^7,35^2)(33^5,31)(29^8,27^5)(25^{11},23^4)$, and
\begin{align*}
\tau^3(\Gamma)&=(37^4,35^5)(33^2,31^4)(29^5,27^8)(25^{8},23^7)(21^0,19^{15})(17,15^9)(13,11^8)(9^0,7^{13})(5^3,3^{10})\\
&=(37^4,35^5)(33^2,31^4)(29^5,27^8)(25^{8},23^7)\cup(19^{15},17)(15^9,13)(11^8)\cup(7^{13},5^3)(3^{10}),
\end{align*}
where the multiplicity of each chain of odd length is at least $8$. Secondly, replacing $\Gamma$ by $\Gamma_\varepsilon$, we find that $\varepsilon=5$ and $\Gamma_\varepsilon=(37^7,35^2)$, and
\begin{align*}
\tau^5(\Gamma)&=(37^2,35^7)(33^0,31^6)(29^3,27^{10})(25^{6},23^9)\\
&=(37^2,35^7)\cup(31^6,29^3)(27^{10},25^{6})(23^9),
\end{align*}
where there is only one chain of odd length, whose multiplicity is $6$. Thirdly, replacing $\Gamma$ by the new $\Gamma_\varepsilon$ again, we see that $\varepsilon=7$ and $\Gamma_\varepsilon$ is empty, and
\[\tau^7(\Gamma)=(35^9),\]
which is a chain of odd length with multiplicity $9$. In the above process, the multiplicity of each chain of odd length is always at least $6$. So the original chain $\Gamma$ is a $6$-fixed chain.
\end{example}

\begin{remark}\label{remkfix}
Obviously, a $\delta$-fixed chain is also a $(\delta-1)$-fixed chain.
\end{remark}

We next investigate the properties of $\delta$-fixed chains.

\begin{theorem}\label{lemmulfix}
Assume that $\Gamma$ is a chain with no $1$'s. If $\varepsilon\geq \delta$, then $\Gamma$ is $\delta$-fixed.
\end{theorem}
\begin{proof}
The proof is by induction on the length of $\Gamma$. If $\ell(\Gamma)=1$, then $\Gamma$ must be $(t^{\varepsilon})$ where $t\geq3$. Since $\Gamma_\varepsilon$ is empty and $\tau^{\varepsilon}(\Gamma)=(t-2^{\varepsilon})$, we can clearly see that $\Gamma$ is $\delta$-fixed. Assume that the assertion is true for $\ell(\Gamma)< l$. We now suppose that $\ell(\Gamma)=l$. First, the multiplicity of each chain of odd length in $\tau^{\varepsilon}(\Gamma)$ is at least $\varepsilon$, which is at least $\delta$. Second, if $\Gamma_\varepsilon$ is nonempty, it is a shorter chain with no $1$'s, whose multiplicity is greater than $\varepsilon$. By the induction hypothesis, we know that $\Gamma_\varepsilon$ is $\delta$-fixed. So the result is true.
\end{proof}

\begin{remark}\label{remark}
Even though $\Gamma$ is a $\delta$-fixed chain, $\varepsilon$ may be less than $\delta$. See Example \ref{exafix} for an instance. But if a $\delta$-fixed chain $\Gamma$ has odd length, then $\varepsilon\geq\delta$ since the smallest part of $\tau^\varepsilon(\Gamma)$ occurs exactly $\varepsilon$ times, and is the maximal element of the last block of $\tau^\varepsilon(\Gamma)$.
\end{remark}

\begin{corollary}
If $\Gamma$ is a $\delta$-fixed chain, then each chain of $\tau^\varepsilon(\Gamma)$ is $\delta$-fixed provided that it has no $1$'s.
\end{corollary}
\begin{proof}
First, the multiplicity of each chain of odd length in $\tau^\varepsilon(\Gamma)$ is at least $\delta$. According to Theorem \ref{lemmulfix}, they are $\delta$-fixed. Second, if there exists a chain of even length, then it must be unique, which is in fact $\tau^\varepsilon(\Gamma_\varepsilon)$. Obviously, $\varepsilon(\tau^\varepsilon(\Gamma_\varepsilon))=\varepsilon(\Gamma_\varepsilon)-\varepsilon$. Applying $\varepsilon(\Gamma_\varepsilon)-\varepsilon$ times of $\tau$ to $\tau^\varepsilon(\Gamma_\varepsilon)$ is equivalent to applying $\varepsilon(\Gamma_\varepsilon)$ times of $\tau$ to $\Gamma_\varepsilon$. It follows from the definition of $\delta$-chains that $\Gamma_\varepsilon$ is $\delta$-fixed. So $\tau^\varepsilon(\Gamma_\varepsilon)$ is $\delta$-fixed.
\end{proof}

\begin{corollary}\label{corch}
If $\Gamma$ is a $\delta$-fixed chain, then each chain without $1$'s produced when incessantly performing $\tau$ is $\delta$-fixed.
\end{corollary}

\begin{theorem}\label{lemfixtau}
If $\Gamma$ is a $\delta$-fixed chain, so is for each of $\tau(\Gamma),\tau^2(\Gamma),\ldots,\tau^{\varepsilon-1}(\Gamma)$.
\end{theorem}
\begin{proof}
It is clear that $\varepsilon(\tau(\Gamma))=\varepsilon(\Gamma)-1$, which implies that all chains of odd length produced after applying $\tau$ to $\tau(\Gamma)$ exactly $\varepsilon(\Gamma)-1$ times are those chains of odd length in $\tau^{\varepsilon}(\Gamma)$. Thus, they all have multiplicity of at least $\delta$. In addition, since $(\tau(\Gamma))_\varepsilon=\tau(\Gamma_\varepsilon)$ and $\Gamma_\varepsilon$ is $\delta$-fixed, the above argument still holds. Continuing in this manner shows that $(\tau(\Gamma))_\varepsilon$ is $\delta$-fixed. So $\tau(\Gamma)$ is a $\delta$-fixed chain.

What we have shown is that if $\Gamma$ is $\delta$-fixed, then $\tau(\Gamma)$ is $\delta$-fixed. Continuing this line of reasoning yields the desired result.
\end{proof}

\begin{lemma}\label{lemtauchai}
Suppose that $\Gamma_1$ and $\Gamma_2$ are two $\delta$-fixed chains of even length. If $\varepsilon(\Gamma_1)\geq \delta-1$ and $\Gamma_1\cup\Gamma_2$ is a chain, then $\Gamma_1\cup\Gamma_2$ is a $\delta$-fixed chain.
\end{lemma}
\begin{proof}
Let $\ell=\ell(\Gamma_1)+\ell(\Gamma_2)$, which is even and at least $4$, and let $\Gamma=\Gamma_1\cup\Gamma_2$. We use $\varepsilon_i$ to denote $\varepsilon(\Gamma_i)$ for $i=1,2$, and $\varepsilon$ to denote $\varepsilon(\Gamma)$.

It is easy to verify that the result is true when $\ell=4$. We shall proceed by induction on $\ell$.

If $\varepsilon_1\leq\varepsilon_2$, then $\varepsilon=\varepsilon_1$ and $\Gamma_\varepsilon=(\Gamma_1)_\varepsilon$, which is $\delta$-fixed since $\Gamma_1$ is $\delta$-fixed. Let $\overline{\Gamma}_\varepsilon$ denote the subchain of $\Gamma$ obtained by removing the subchain $\Gamma_\varepsilon$. Then $\overline{\Gamma}_\varepsilon$ is a chain of even length, where the multiplicity of the maximal element in each block is at least $\varepsilon_1$, and the maximal element of the first block has multiplicity of exactly $\varepsilon_1$. Thus, $\varepsilon(\overline{\Gamma}_\varepsilon)=\varepsilon_1=\varepsilon$. Furthermore, each chain of $\tau^{\varepsilon}(\overline{\Gamma}_\varepsilon)$ has odd length with multiplicity being at least $\varepsilon+1$, which is at least $\delta$. They are exactly the chains of odd length in $\tau^\varepsilon(\Gamma)$.

If $\varepsilon_1>\varepsilon_2$, then $\varepsilon=\varepsilon_2$, and $\Gamma_\varepsilon=\Gamma_1\cup(\Gamma_2)_\varepsilon$, which is $\delta$-fixed by the induction hypothesis.
The chains of odd length in $\tau^\varepsilon(\Gamma)$ are exactly those in $\tau^{\varepsilon}(\Gamma_2)$, each of which has multiplicity of at least $\delta$ since $\Gamma_2$ is $\delta$-fixed.

Therefore, $\Gamma$ is always a $\delta$-fixed chain.
\end{proof}

\begin{theorem}\label{Corchmer}
Suppose that $\Gamma_1$ and $\Gamma_2$ are two $\delta$-fixed chains where $\delta\geq2$ and $\Gamma_1$ has odd length. If $\tau(\Gamma_1)\cup\tau(\Gamma_2)$ is a chain without $1$'s, then $\tau(\Gamma_1)\cup\tau(\Gamma_2)$ is a $\delta$-fixed chain.
\end{theorem}
\begin{proof}
It follows from Remark \ref{remark} that $\varepsilon(\Gamma_1)\geq \delta$. Therefore, $\varepsilon(\tau(\Gamma_1))\geq \delta-1$. Both $\tau(\Gamma_1)$ and $\tau(\Gamma_2)$ are $\delta$-fixed according to Theorem \ref{lemfixtau}, and clearly have even length. We now see that the desired result is a direct consequence of Lemma \ref{lemtauchai}.
\end{proof}

\begin{theorem}\label{thmdelta}
For an odd partition $\lambda$, if each chain without $1$'s in $\lambda$ is $\delta$-fixed, then each chain without $1$'s in $\tau^{r}(\lambda)$ is $\delta$-fixed for all $r\geq 1$.
\end{theorem}
\begin{proof}
It suffices to prove the result for $r=1$. The case $\delta=1$ is trivial. So we assume that $\delta\geq2$.

Let $\Gamma$ be a $\delta$-fixed chain in $\lambda$. Then each chain without $1$'s in $\tau(\Gamma)$ is $\delta$-fixed according to Corollary \ref{corch}. Such chains are separated from $\Gamma$ when performing $\tau$.

There exist another kind of chains in $\tau(\lambda)$, which are produced by merging two chains. The former chain must be $\tau(\Gamma_1)$, where $\Gamma_1$ is a chain of odd length in $\lambda$; the latter one must either be $\tau(\Gamma)$ or $\tau(\Gamma_\varepsilon)$ where $\Gamma$ is a chain of $\lambda$. We can always denote the latter chain by $\tau(\Gamma_2)$ where $\Gamma_2$ is $\delta$-fixed. It follows from Theorem \ref{Corchmer} that $\tau(\Gamma_1)\cup\tau(\Gamma_2)$ is $\delta$-fixed.

Each chain of $\tau(\lambda)$ is either separated from a chain of $\lambda$ or merged by two chains under the operation $\tau$. The above arguments show that it is always $\delta$-fixed provided it contains no $1$'s.
\end{proof}

For a chain $\Gamma$, the block pattern of $\Gamma$ refers to the way we group the distinct parts of $\Gamma$ into blocks. If the two distinct elements in every block of $\Gamma$ are fixed when performing $\tau$, i.e., the companion of each element is changeless, we say that the block pattern of $\Gamma$ remains the same. Let $t$ be the largest element of $\Gamma$, which lies in a block of $\Gamma$ with $t-2$. If $t+2$ and $t$ occupy a block together after performing $\tau$ several times, we say that the block pattern of $\Gamma$ will be affected by the chain before it. Moreover, the chain affecting $\Gamma$ must have odd length. If the way we group the elements of $\Gamma$ into blocks is irrelevant to the elements before $\Gamma$, we say that the block pattern of $\Gamma$ will not be affected by the chains before it.

\begin{lemma}\label{lemsp}
Assume that $\Gamma$ is a $\delta$-fixed chain with $t$ as its smallest element. For $0<r<\delta$, the smallest element of $\tau^{r}(\Gamma)$ is either $t$ or $t-2$. Moreover, if the smallest element of $\tau^r(\Gamma)$ is $t-2$, then $t$ and $t-2$ occupy a block together.
\end{lemma}
\begin{proof}
We prove the result according to the parity of the length of $\Gamma$.

Case 1. The length of $\Gamma$ is odd. It follows from Remark \ref{remark} that $\varepsilon\geq\delta$, so $\varepsilon>r$. Thus, $\tau^r(\Gamma)$ is still a full chain, i.e., the chain $\Gamma$ was not disconnected under the operation $\tau$. Since $\ell(\Gamma)$ is odd, $t$ inhabits a block by itself in $\Gamma$, implying $t-2$ is the smallest element of $\tau^r(\Gamma)$ and lies in a block with $t$ together.

Case 2. The length of $\Gamma$ is even. In this case, $t+2$ and $t$ share a block in $\Gamma$. Clearly, $\Gamma$ will be disconnected for the first time after performing exactly $\varepsilon$ times of the operation $\tau$. The smallest element of $\tau^{\varepsilon}(\Gamma)$ is also $t$, occurring at least $\varepsilon+1$ times. If $r\leq\varepsilon$, then the result is obviously true. We next assume that $r>\varepsilon$. We claim that if $\ell(\Gamma)$ is even and $r>\varepsilon$, then $t-2$ is the smallest element of $\tau^r(\Gamma)$ and lies in a block with $t$ together. We prove by induction on $\ell(\Gamma)$. When $\ell(\Gamma)=2$, we assume that $\Gamma=(t+2^x,t^y)$ where $x\geq 1$ and $y\geq1$. Since $r>\varepsilon=x$, we see that $\tau^r(\Gamma)=(t^{2x+y-r},t-2^{r-x})$. Thus, the claim is true for $\ell(\Gamma)=2$. Let $\ell(\Gamma)\geq4$, and assume the chain decomposition of $\tau^\varepsilon(\Gamma)$ is
\[\tau^\varepsilon(\Gamma)=\Gamma_1\cup\Gamma_2\cup\cdots\cup\Gamma_k.\]
Then $\Gamma_1=\tau^{\varepsilon}(\Gamma_\varepsilon)$, which is the unique chain of even length in $\tau^\varepsilon(\Gamma)$, and the smallest element of $\Gamma_k$ is $t$. It follows from Theorem \ref{thmdelta} that each $\Gamma_i$ is $\delta$-fixed. For $1\leq i\leq k$, let $a_i$ and $b_i$ denote the largest and smallest element of $\Gamma_i$, respectively. Then $b_i=a_{i+1}+4$ for $1\leq i<k$. For $1\leq i\leq k$, let $\bar{a}_i$ and $\bar{b}_i$ denote the largest and smallest element of $\tau^{r-\varepsilon}(\Gamma_i)$, respectively. Note that $\Gamma_1$ is a shorter chain of even length. If $r-\varepsilon\leq\varepsilon(\Gamma_1)$, then $\bar{b}_1=b\geq\bar{a}_2+4$; otherwise by the induction hypothesis we know that $\bar{b}_1$ equals $b_1-2$ and lies in a block with $b_1$ together. Thus, we can always obtain that $\bar{b}_1$ will not stay in a block with $\bar{a}_2$ together. For $2\leq i\leq k$, since each $\Gamma_i$ is $\delta$-fixed and has odd length, we have $\varepsilon(\Gamma_i)\geq\delta$. Because $\delta>r>r-\varepsilon$, we get that $\bar{b}_i=b_i-2$ and stays in a block with $b_i$ in $\tau^{r-\varepsilon}(\Gamma_i)$ for $2\leq i\leq k$. Thus, $\bar{b}_i$ will not in a block with $\bar{a}_{i+1}$ for $2\leq i<k$. Therefore, the ways of grouping the elements of $\tau^{r-\varepsilon}(\Gamma_i)$ are independent. Because $t$ and $t-2$ build the last block of $\tau^{r-\varepsilon}(\Gamma_k)$ and $\tau^r(\Gamma)=\tau^{r-\varepsilon}(\tau^\varepsilon(\Gamma))$, the claim follows.
\end{proof}

We now come to an important property of $\delta$-fixed chains, which plays a very crucial role in our later proofs.

\begin{theorem}\label{thmnaff}
Assume that $\Gamma$ is a chain of $\lambda$ and all chains before $\Gamma$ are $\delta$-fixed. Then the block pattern of $\Gamma$ will not be affected by the elements before it provided that the operation $\tau$ is performed less than $\delta$ times.
\end{theorem}
\begin{proof}
Assume that the chains before $\Gamma$ are $\Gamma_1,\Gamma_2,\ldots,\Gamma_k$, and $0<r<\delta$. Suppose that the smallest element of $\Gamma_k$ is $a$, and the largest element of $\Gamma$ is $b$. Clearly, $a\geq b+4$. It follows from Lemma \ref{lemsp} that the ways of grouping the elements of $\tau^r(\Gamma_i)$ are independent.

Case 1. If the smallest element of $\tau^r(\Gamma_k)$ is $a$, then $a$ cannot be in a block with the largest element of $\tau^r(\Gamma)$ in $\tau^r(\lambda)$ because $a\geq b+4$.

Case 2. If the smallest element of $\tau^r(\Gamma_k)$ is $a-2$, then $a$ and $a-2$ will stay in a block of $\tau^r(\Gamma_k)$ according to Lemma \ref{lemsp}. So do in $\tau^r(\lambda)$.

Thus the block pattern of $\Gamma$ will not be affected by the elements before it.
\end{proof}

\section{Upper bound}\label{secub}

\begin{lemma}\label{thmfirstchain}
Let $\lambda$ be an odd partition whose largest part occurs exactly $r$ times. If the first chain of $\tau^r(\lambda)$ does not contain $1$'s, then its length must be odd.
\end{lemma}
\begin{proof}
Suppose that the chain decomposition of $\lambda$ is $\Gamma_1\cup\Gamma_2\cup\cdots\cup\Gamma_k$. For $1\leq i\leq k$, we use ${\Gamma}_i^r$ to denote the first chain of $\tau^r(\Gamma_i)$.

There are two cases to discuss.

If $k>1$, then $\ell(\Gamma^r_1)$ is odd. The first chain of $\tau^r(\lambda)$ must be $\Gamma^r_1\cup\Gamma^r_2\cup\cdots\cup\Gamma^r_j$ for some $j\geq1$. Each $\ell(\Gamma^r_i)$ ($2\leq i\leq j$) must be even otherwise the largest part of $\Gamma_i$ does not appear in $\Gamma^r_i$, contradicting to that $\Gamma^r_1\cup\Gamma^r_2\cup\cdots\cup\Gamma^r_i$ is a chain. Thus, the first chain of $\tau^r(\lambda)$ has odd length.

If $k=1$, then $\lambda$ has only one chain, i.e., $\lambda=\Gamma_1$. So the first chain of $\tau^r(\lambda)$ is $\Gamma^r_1$. If $\lambda$ has no part of size $1$, then $\ell(\Gamma^r_1)$ is odd. We next suppose that $\lambda$ has parts of size $1$. If $\Gamma_1$ will be destroyed during the process of performing the operation $\tau$, then $\Gamma^r_1$ has odd length; otherwise $\Gamma^r_1$ is a chain containing $1$'s. In this subcase, $\ell(\Gamma^r_1)$ may be even.

Thus, the length of the first chain of $\tau^r(\lambda)$ may be even only when it contains $1$'s.
\end{proof}

\begin{lemma}\label{thmkeep}
Assume that $\lambda$ is an odd partition with chain decomposition $\lambda=\Gamma_1\cup\Gamma_2\cup\cdots\cup\Gamma_k$ where $\ell(\Gamma_j)$ is odd, $\varepsilon(\Gamma_j)=\delta$ and each $\Gamma_i$ is $\delta$-fixed for $1\leq i<j $. Letting $s$ be the smallest element of $\Gamma_j$, then the chain containing $s-2$ in $\tau^\delta(\lambda)$ has multiplicity of exactly $\delta$, and has odd length if it does not contain $1$'s. Furthermore, the chains before it in $\tau^\delta(\lambda)$ are still $\delta$-fixed.
\end{lemma}
\begin{proof}
Since $\varepsilon(\Gamma_j)=\delta$ and $\ell(\Gamma_j)$ is odd, we know that $\Gamma_j$ is $\delta$-fixed according to Theorem \ref{lemmulfix}. It follows from Theorem \ref{thmnaff} that the block pattern of $\Gamma_j$ will not be affected before we perform the $\delta$-th operation of $\tau$. Suppose that the chain $\Gamma_j$ is of the following form
\begin{align*}
\Gamma_j=\cdots(r+4^{\geq \delta},r+2^{\geq1})(r^{=\delta},r-2^{\geq1})(r-4^{>\delta},r-6^{\geq1})
\cdots(s+4^{>\delta},s+2^{\geq1})(s^{>\delta}),
\end{align*}
where $r$ is the last element of $\Gamma_j$ as the maximal element of a block and occurring exactly $\delta$ times. Clearly, $\Gamma_j$ will be disconnected in the location of $r$. From Theorem \ref{thmdelta}, we know that the elements before $r$ will form new $\delta$-fixed chains in $\tau^{\delta}(\lambda)$. The remaining segment of $\Gamma_j$ will become
\begin{align*}
\overline{\Gamma}_j&=(r^{=0},r-2^{>\delta})(r-4^{>0},r-6^{>\delta})
\cdots(s+4^{>0},s+2^{>\delta})(s^{>0},s-2^{=\delta}),\\
&=(r-2^{>\delta},r-4^{\geq1})(r-6^{>\delta},r-8^{\geq1})
\cdots(s+6^{>\delta},s+4^{\geq1})(s+2^{>\delta},s^{\geq1})(s-2^{=\delta}),
\end{align*}
in $\tau^{\delta}(\lambda)$.

If $s-4$ does not appear in $\tau^\delta(\lambda)$, then $\overline{\Gamma}_j$ is the chain containing $s-2$, whose length is odd and multiplicity is exactly $\delta$.

If $s-4\neq 1$ and appears in $\tau^\delta(\lambda)$, which implies that $s-4$ must be the largest element of $\Gamma_{j+1}$, then $s-6$ also appears in $\tau^\delta(\lambda)$ and appears at least $\delta$ times. Similarly, if $s-8\neq1$ and appears, then $s-10$ appears at least $\delta$ times. Let $\overline{\Gamma}_{j+1}$ be the subchain beginning with $s-4$ in $\tau^\delta(\lambda)$. Let $t$ be the last element of $\overline{\Gamma}_{j+1}$. Then the chain of $\tau^\delta(\lambda)$ containing $s-2$ is $\overline{\Gamma}_j\cup\overline{\Gamma}_{j+1}$. If $t\neq1$, then $\ell(\overline{\Gamma}_{j+1})$ is even. In this case,
$\overline{\Gamma}_j\cup\overline{\Gamma}_{j+1}$ can be written as
\begin{align*}
\cdots(s+2^{>\delta},s^{\geq1})(s-2^{=\delta},s-4^{\geq1})(s-6^{\geq \delta},s-8^{\geq1})\cdots(t+4^{\geq \delta},t+2^{\geq1})(t^{\geq \delta}),
\end{align*}
whose length is odd and multiplicity is $\delta$.
\end{proof}

For a chain $\Gamma$, we use $\kappa(\Gamma)$ to denote the last element of $\Gamma$ which is the maximal element of the block containing it and occurs exactly $\varepsilon(\Gamma)$ times in $\Gamma$.

\begin{theorem}\label{thmflhleq}
If the odd partition $\lambda$ with largest part $\lambda_1$ corresponds to the distinct partition $\pi$ with lecture hall length $\ell_h(\pi)$ under our bijection $\varphi$, then we have $(\lambda_1+1)/2\leq\ell_h(\pi)$.
\end{theorem}
\begin{proof}
Suppose that the largest part of $\lambda$ is $2t-1$ and occurs exactly $r$ times in $\lambda$.

Set $\Phi_0=\tau^r(\lambda)$ and let $\Upsilon_0$ denote the first chain of $\Phi_0$. Suppose that $\varepsilon(\Upsilon_0)=\delta$, and there are $a_0$ blocks before $\kappa(\Upsilon_0)$ in $\Upsilon_0$. Then
$\kappa(\Upsilon_0)=2t-4a_0-3$ since the first element of $\Phi_0$ is clearly $2t-3$. Assume that the last element of $\Upsilon_0$ is $b_0$. Note that if $b_0\neq\kappa(\Upsilon_0)$, then it occurs more than $\delta$ times in $\Phi_0$.

If $\Upsilon_0$ does not contain $1$'s, then its length is odd due to Lemma \ref{thmfirstchain}. Set $\Phi_1=\tau^\delta(\Phi_0)$. Now we see that $b_0-2$ occurs exactly $\delta$ times in $\Phi_1$, and becomes the maximal element of some block because $2t-4a_0-3$ disappears in $\Phi_1$. Let $\Upsilon_1$ denote the chain of $\Phi_1$ containing the element $b_0-2$, the first element of which must be $2t-4a_0-5$. Lemma \ref{thmkeep} tells us that $\varepsilon(\Upsilon_1)=\delta$ and the chains before $\Upsilon_1$ are $\delta$-fixed. Assume that there are $a_1$ blocks before $\kappa(\Upsilon_1)$ in $\Upsilon_1$. Then $\kappa(\Upsilon_1)=2t-4a_0-4a_1-5$. Assume that the last element of $\Upsilon_1$ is $b_1$. Note that $b_1$ (if $b_1\neq\kappa(\Upsilon_1)$) occurs more than $\delta$ times in $\Phi_1$.

If $\Upsilon_1$ does not contain $1$'s, then $\ell(\Upsilon_1)$ is odd according to Lemma \ref{thmkeep}. Set $\Phi_2=\tau^{\delta}(\Phi_1)$. Because the chains before $\Upsilon_1$ are $\delta$-fixed, Theorem \ref{thmnaff} tells us that the block pattern of $\Upsilon_1$ will not be affected during the process of producing $\Phi_2$. Thus, $b_1-2$ occurs exactly $\delta$ times in $\Phi_2$, and becomes the maximal element of some block because $2t-4a_0-4a_1-5$ disappears in $\Phi_2$. Let $\Upsilon_2$ denote the chain of $\Phi_2$ containing $b_1-2$, the first element of which must be $2t-4a_0-4a_1-7$. According to Lemma \ref{thmkeep}, we know that $\varepsilon(\Upsilon_2)=\delta$ and the chains before $\Upsilon_2$ are $\delta$-fixed. Assume that there are $a_2$ blocks before $\kappa(\Upsilon_2)$ in $\Upsilon_2$. Then we have $\kappa(\Upsilon_2)=2t-4a_0-4a_1-4a_2-7$. Assume that the last element of $\Upsilon_2$ is $b_2$. Note that $b_2$ (if $b_2\neq\kappa(\Upsilon_2)$) occurs more than $\delta$ times in $\Phi_2$.

If $\Upsilon_2$ does not contain $1$'s, then $\ell(\Upsilon_2)$ is odd according to Lemma \ref{thmkeep}. We continue the above process until eventually encountering the case that $\Upsilon_p$ is a chain containing $1$'s.

We produce two sequences, one of which is formed by partitions $(\Phi_0,\Phi_1,\ldots,\Phi_{p-1},\Phi_p)$ and the other is formed by chains $(\Upsilon_0,\Upsilon_1,\ldots,\Upsilon_{p-1},\Upsilon_p)$ satisfying  $\Phi_{i+1}=\tau^\delta(\Phi_i)$ and $\Upsilon_{i+1}$ is the chain of $\Phi_{i+1}$ containing the element $b_i-2$ where $b_i$ is the last element of $\Upsilon_{i}$.

We now come back to analyze $\Upsilon_{p-1}$. We know that $\Upsilon_{p-1}$ is the chain of $\Phi_{p-1}$ containing the element $b_{p-2}-2$, where $b_{p-2}$ is the last element of $\Upsilon_{p-2}$. The first element of $\Upsilon_{p-1}$ is
\[2t-4\sum\limits_{i=0}^{p-2}a_i-(2p+1).\]
Owing to Lemma \ref{thmkeep}, we know that $\varepsilon(\Upsilon_{p-1})=\delta$, the length of $\Upsilon_{p-1}$ is odd, and the chains before $\Upsilon_{p-1}$ are $\delta$-fixed. Since there are $a_{p-1}$ blocks before $\kappa(\Upsilon_{p-1})$ in $\Upsilon_{p-1}$, we obtain
\[\kappa(\Upsilon_{p-1})=2t-4\sum\limits_{i=0}^{p-1}a_i-(2p+1).\]
Note that $b_{p-1}$, the last element of $\Upsilon_{p-1}$, occurs more than $\delta$ times in $\Phi_{p-1}$ if $b_{p-1}\neq\kappa(\Upsilon_{p-1})$. So $b_{p-1}-2$ occurs exactly $\delta$ times in $\Phi_{p}$.

The chain $\Upsilon_p$ we previously encountered is indeed the chain of $\Phi_p$ containing $b_{p-1}-2$. We now assume that
\begin{align*}
\Upsilon_{p-1}=\cdots
(\kappa(\Upsilon_{p-1})^{=\delta},\kappa(\Upsilon_{p-1})-2^{\geq1})(\kappa(\Upsilon_{p-1})-4^{>\delta},\kappa(\Upsilon_{p-1})-6^{\geq1})\cdots
(b_{p-1}^{>\delta}).
\end{align*}
To ensure that $1$ appears in $\Upsilon_p$, the chain $\Upsilon_{p-1}$ must be one the following two cases.

Case 1. If $b_{p-1}=3$, then
\begin{align*}
\Upsilon_{p}=(\kappa(\Upsilon_{p-1})-2^{>\delta},\kappa(\Upsilon_{p-1})-4^{\geq1})\cdots(5^{>\delta},3^{\geq1})(1^\delta).
\end{align*}

Case 2. If $b_{p-1}>3$, the block pattern of the parts smaller than $b_{p-1}-2$ in $\Phi_{p-1}$ must be either
\begin{align*}
(b_{p-1}-4^{>\delta},b_{p-1}-6^{\geq 0})(b_{p-1}-8^{>\delta},b_{p-1}-10^{\geq 0})\cdots(3^{>\delta},1^{\geq0})
\end{align*}
or
\begin{align*}
(b_{p-1}-4^{>\delta},b_{p-1}-6^{\geq 0})(b_{p-1}-8^{>\delta},b_{p-1}-10^{\geq 0})\cdots(5^{>\delta},3^{\geq0})(1^{>\delta}).
\end{align*}
Correspondingly, we have
\begin{align*}
\Upsilon_p=(\kappa(\Upsilon_{p-1})-2^{>\delta},\kappa(\Upsilon_{p-1})-4^{\geq1})\cdots(b_{p-1}-2^{=\delta},b_{p-1}-4^{\geq 1})\cdots(5^{\geq\delta},3^{\geq 1})(1^{\geq\delta})
\end{align*}
or
\begin{align*}
\Upsilon_p=(\kappa(\Upsilon_{p-1})-2^{>\delta},\kappa(\Upsilon_{p-1})-4^{\geq1})\cdots(b_{p-1}-2^{=\delta},b_{p-1}-4^{\geq 1})\cdots(7^{\geq\delta},5^{\geq 1})(3^{\geq\delta},1^{\geq 1}).
\end{align*}

We now see that the first element of $\Upsilon_p$ is
\[\kappa(\Upsilon_{p-1})-2=2t-4\sum\limits_{i=0}^{p-1}a_i-(2p+3).\]
Assume that there are $a_p$ blocks before $b_{p-1}-2$ in $\Upsilon_p$. Then
\begin{align}\label{eqb1}
b_{p-1}-2=2t-4\sum\limits_{i=0}^{p}a_i-(2p+3).
\end{align}

Suppose that $b_o(\Phi_{p-1})=j+1$. It is easy to see that each chain of $\Phi_{p-1}$ is $\delta$-fixed, implying the block odd index of $\Phi_{p-1}$ remains the same if the operation $\tau$ was performed less than $\delta$ times. Namely,
\begin{align*}
b_o(\Phi_{p-1})=b_o(\tau(\Phi_{p-1}))=b_o(\tau^2(\Phi_{p-1}))=\cdots=b_o(\tau^{\delta-1}(\Phi_{p-1})).
\end{align*}
It is clear that $\Upsilon_{p-1}$ will be disconnected in the location of $\kappa(\Upsilon_{p-1})$ in $\tau^\delta(\Phi_{p-1})$, causing the block pattern for the parts less than it to be changed. The above two cases of $\Upsilon_{p-1}$ implies that the parity of $b_o(\Phi_p)$ and $b_o(\Phi_{p-1})$ is different. Because $b_o(\tau^{\delta-1}(\Phi_{p-1}))-1\leq b_o(\Phi_p)\leq b_o(\tau^{\delta-1}(\Phi_{p-1}))$, we conclude that $b_o(\Phi_p)=j$.

We now consider $\Phi_{p+1}=\tau^\delta(\Phi_p)$. Similarly, since each chain of $\Phi_p$ is $\delta$-fixed, we see that
\begin{align*}
b_o(\Phi_{p})=b_o(\tau(\Phi_{p}))=b_o(\tau^2(\Phi_{p}))=\cdots=b_o(\tau^{\delta-1}(\Phi_{p})).
\end{align*}
and $b_o(\tau^{\delta-1}(\Phi_{p}))-1\leq b_o(\Phi_{p+1})\leq b_o(\tau^{\delta-1}(\Phi_{p}))$. Because $\Upsilon_p$ will be disconnected in the location of $b_{p-1}-2$, affecting the block pattern behind it, we obtain $b_o(\Phi_{p+1})=b_o(\Phi_{p})-1=j-1$. We now conclude that as a weight, $j$ occurs exactly $\delta$ times in the gap-free sequence $\mathbf{s}$ associated with $\lambda$.

The block pattern of $\Upsilon_p$ is
\[\cdots(b_{p-1}+2^{>\delta},{b_{p-1}}^{\geq 1})(b_{p-1}-2^{=\delta},b_{p-1}-4^{\geq 1})(b_{p-1}-6^{\geq\delta},b_{p-1}-8^{\geq 1})\cdots.\]
Since the chains involved in the previous steps are always $\delta$-fixed, the sum of the numbers of their blocks will not be changed. Thus, there are $a_0+a_1+\cdots+a_p$ blocks before $b_{p-1}-2$ in $\Phi_p$. We deduce that the block odd index of $\Phi_p$ is
\begin{align}\label{eqb2}
j=2\sum\limits_{i=0}^pa_i+\frac{b_{p-1}-2+1}{2}=2\sum\limits_{i=0}^pa_i+\frac{b_{p-1}-1}{2}
\end{align}
Combining \eqref{eqb1} and \eqref{eqb2} together, we arrive at
\begin{align}\label{eqodd}
t=j+p+1.
\end{align}

To produce $\Phi_{p+1}$, we perform totally $(p+1)\delta+r$ times of the operation $\tau$. Put another way, $\tau$ has been performed $(p+1)\delta+r$ times in total before we get the weight $j-1$ in the gap-free sequence $\mathbf{s}$. Thus, $\pi_j=(p+1)\delta+r$. In addition, the multiplicity of $j$ in $\mathbf{s}$ is $\delta$, i.e., $\pi_j-\pi_{j+1}=\delta$. We can now derive the conclusion
\begin{align}\label{eqlh}
\ell_h(\pi)&\geq\left\lceil\frac{\pi_j}{\pi_j-\pi_{j+1}}\right\rceil+j-1\nonumber\\
&=\left\lceil\frac{(p+1)\delta+r}{\delta}\right\rceil+j-1\nonumber\\
&=p+j+1.
\end{align}

Combining \eqref{eqodd} and \eqref{eqlh}, we see that $t\leq\ell_h(\pi)$.

If $\Phi_0$ contains $1$'s in the initial step, since $\varepsilon(\Phi_0)=\delta$, we see that $b_o(\Phi_0)$ occurs exactly $\delta$ times in the gap-free sequence $\mathbf{s}$. Using similar arguments, we can obtain the desired result. We can also regard this situation as the special case for $p=0$.
\end{proof}

\begin{example}
Taking $\lambda=(31^2,29^3,27^5,25^6,21^4,19^4,17,15^8,11^4,9,7^5,5^8,3^6,1^8)$, then $t=16$.
Since the largest part of $\lambda$ occurs twice, we compute \[\tau^2(\lambda)=(29^5,27^3,25^8,21^2,19^6,15^8,13,11^2,9^3,7^3,5^{10},3^4,1^{10}),\] the first chain of which is $(29^5,27^3,25^8)=(29^5,27^3)(25^8)$. Thus, $\delta=5$. See Figure \ref{figcom} for the computation of $\Phi_i$.
\begin{figure}[ht]
\begin{center}
\begin{tikzpicture}
\draw (0,0)--(15,0); \draw (0,0.6)--(15,0.6); \draw (0,1.2)--(15,1.2); \draw (0,1.8)--(15,1.8);
\draw (0,2.4)--(15,2.4); \draw (0,3)--(15,3); \draw (0,3.6)--(15,3.6); \draw (0,4.2)--(15,4.2);

\node at (0.15,3.9) {$i$};\node at (0.15,3.3) {$0$};\node at (0.15,2.7) {$1$};\node at (0.15,2.1) {$2$};
\node at (0.15,1.5) {$3$};\node at (0.15,0.9) {$4$};\node at (0.15,0.3) {$5$};

\node at (4.7,3.9) {$\Phi_i$};
\node at (4.7,3.3) {$(29^5,27^3,25^8,21^2,19^6,15^8,13,11^2,9^3,7^3,5^{10},3^4,1^{10})$};
\node at (4.7,2.7) {$(27^8,25^3,23^5,19^5,17^3,15^3,13^6,9^2,7^4,5^{9},3^5,1^{9})$};
\node at (4.7,2.1) {$(27^3,25^8,21^5,17^8,13^7,11^2,7^3,5^{10},3^4,1^{10})$};
\node at (4.7,1.5) {$(25^9,23^2,19^5,17^3,15^5,13^2,11^7,5^{11},3^3,1^{11})$};
\node at (4.7,0.9) {$(25^4,23^7,17^8,13^7,11^2,9^5,5^{6},3^8,1^{6})$};
\node at (4.7,0.3) {$(23^{10},21,17^3,15^5,13^2,11^7,7^5,5,3^{13},1)$};

\node at (11.2,3.9) {$\Upsilon_i$};
\node at (11.2,3.3) {$(29^5,27^3)(25^8)$};
\node at (11.2,2.7) {$(27^8,25^3)(23^5)$};
\node at (11.2,2.1) {$(21^5)$};
\node at (11.2,1.5) {$(19^5,17^3)(15^5,13^2)(11^7)$};
\node at (11.2,0.9) {$(13^7,11^2)(9^5)$};
\node at (11.2,0.3) {$(7^5,5)(3^{13},1)$};

\node at (13.9,3.9) {$\kappa(\Upsilon_i)$};
\node at (13.9,3.3) {$29$};\node at (13.9,2.7) {$23$};\node at (13.9,2.1) {$21$};
\node at (13.9,1.5) {$15$};\node at (13.9,0.9) {$9$};\node at (13.9,0.3) {$7$};

\node at (14.7,3.9) {$a_i$};\node at (14.7,3.3) {$0$};\node at (14.7,2.7) {$1$};
\node at (14.7,2.1) {$0$};\node at (14.7,1.5) {$1$};\node at (14.7,0.9) {$1$};
\node at (14.7,0.3) {$0$};

\draw (0,0)--(0,4.2);\draw (0.3,0)--(0.3,4.2);\draw (9.1,0)--(9.1,4.2);\draw (13.4,0)--(13.4,4.2);
\draw (14.4,0)--(14.4,4.2);\draw (15,0)--(15,4.2);
\end{tikzpicture}\caption{The computation of $\Phi_i$.}\label{figcom}
\end{center}
\end{figure}
We now see that $p=5$ and $j=b_o(\Phi_5)=10$, satisfying $t=j+p+1$. Moreover, $j=10$ as a weight must occur $\delta=5$ times in the gap-free sequence $\mathbf{s}$ associated with $\lambda$. This is actually true since \[\mathbf{s}=(1^{11},2^{47},3^{10},4^9,5^{11},6^8,7^{10},8^{14},9^6,{10^5},11^{12},12^6,13^5,14^4).\]
In addition, we have
\[\pi=\varphi(\lambda)=\mathbf{s}'=(158,147,100,90,81,70,62,52,38,32,27,15,9,4),\]
where $\pi_j=\pi_{10}=32$ satisfying $\pi_j=(p+1)\delta+r$. We now obtain
\[\ell_h(\pi)\geq\left\lceil\frac{\pi_{10}}{\pi_{10}-\pi_{11}}\right\rceil+9=\left\lceil\frac{32}{32-27}\right\rceil+9
=16=j+p+1=t.\]
\end{example}

In fact, we can extract the following result from the proof of Theorem \ref{thmflhleq}. For convenience and simplicity, we use $g(\Gamma)$ to denote the largest element of a chain $\Gamma$ in the rest of the paper.
\begin{theorem}\label{thmmapkey}
Suppose that $\lambda$ is an odd partition with $\mathbf{s}=(1^{m_1},2^{m_2},\ldots,k^{m_k})$ as its associated gap-free sequence. Let $\Gamma$ be a chain with exactly $a$ blocks before it and all chains before it being $\delta$-fixed. If $\Gamma$ is $\delta$-fixed but not $(\delta+1)$-fixed, then $\delta=m_j$ for some $j<k$ and
\begin{align*}
g(\Gamma)=2j+2\left\lceil(m_{j+1}+m_{j+2}+\cdots+m_k)/m_j\right\rceil-4a-1.
\end{align*}
If $\Gamma$ includes $1$'s and $\varepsilon(\Gamma)=\delta$, then $\delta=m_k$ and $g(\Gamma)=2k-4a-1$.
\end{theorem}
\begin{proof}
Recall that the notion ``$\delta$-fixed'' is valid only for chains without $1$'s. Thus, we prove the result according to whether $\Gamma$ contains $1$'s or not.

Case 1. The chain $\Gamma$ excludes $1$'s.

If $\Gamma$ has odd length, it follows from Remark \ref{remark} that $\varepsilon(\Gamma)=\delta$. In the proof of Theorem \ref{thmflhleq}, we set $\Phi_0=\lambda$ and $\Upsilon_0=\Gamma$, and assume $g(\Gamma)=2t-3$ and preserve all other notation.

Lemma \ref{thmkeep} tells us that if $\Upsilon_i$ has odd length with $\varepsilon(\Upsilon_i)=\delta$ and all chains before $\Upsilon_i$ are $\delta$-fixed, then $\varepsilon(\Upsilon_{i+1})=\delta$ and the chains before $\Upsilon_{i+1}$ are $\delta$-fixed. Furthermore, if $\Upsilon_{i+1}$ does not contain $1$'s, the length of $\Upsilon_{i+1}$ is odd. So the arguments used in the proof of Theorem \ref{thmflhleq} are still valid. In addition, we clearly see that
\[j=b_o(\Phi_p)<b_o(\Phi_{p-1})\leq b_o(\Phi_0)=k.\]
Hence, as a weight, $j$ occurs exactly $\delta$ times in the gap-free sequence $\mathbf{s}$. So $\delta=m_j$ and $j<k$. Since every time we perform $m_j$ times of the operation $\tau$, we conclude that
\begin{align*}
m_k+m_{k-1}+\cdots+m_{j+1}=pm_j,
\end{align*}
which is a multiple of $m_j$.

Take note that in the proof of Theorem \ref{thmflhleq}, since we start with the first chain, there is no block before it. But this time we begin with a designated chain $\Gamma$, so we should take into account the number of blocks before $\Gamma$ when we compute the block odd index of $\Phi_p$. Since there are $a+a_0+a_1+\cdots+a_p$ blocks before $b_{p-1}-2$ in $\Phi_p$, we deduce that the block odd index of $\Phi_p$ is
\begin{align}\label{eqb2new}
j=2\left(a+\sum\limits_{i=0}^pa_i\right)+\frac{b_{p-1}-2+1}{2}=2\left(a+\sum\limits_{i=0}^pa_i\right)+\frac{b_{p-1}-1}{2}
\end{align}
Combining \eqref{eqb1} and \eqref{eqb2new} together, we arrive at
\begin{align*}
t=j+p-2a+1.
\end{align*}
Because $g(\Gamma)=2t-3$, we see that
\begin{align*}
g(\Gamma)&=2j+2(m_{j+1}+m_{j+2}+\cdots+m_k)/m_j-4a-1\\
&=2j+2\left\lceil(m_{j+1}+m_{j+2}+\cdots+m_k)/m_j\right\rceil-4a-1.
\end{align*}

If $\Gamma$ has even length, there must exist an integer $r$ with $0<r<\delta$ such that a chain $\overline{\Gamma}$ of odd length with multiplicity of $\delta$ will be separated from $\Gamma$ in $\tau^r(\lambda)$ otherwise $\Gamma$ is $(\delta+1)$-fixed. Because the chains before $\Gamma$ are all $\delta$-fixed, $\overline{\Gamma}$ is a full chain not a subchain of $\tau^r(\lambda)$. Since $\Gamma$ does not contain $1$'s and $\Gamma$ has even length, the chain $\overline{\Gamma}$ must exclude $1$'s. Furthermore, all chains before $\overline{\Gamma}$ in $\tau^r(\lambda)$ are $\delta$-fixed. We this time set $\Phi_0=\tau^r(\lambda)$ and $\Upsilon_0=\overline\Gamma$, and assume $g(\overline\Gamma)=2t-3$. Applying similar arguments, we conclude that $\delta$ is the number of occurrences of some weight in the gap-free sequence associated with $\tau^r(\lambda)$, denoted ${\mathbf{s}}_\tau$. Note that ${\mathbf{s}}_\tau$ is a subsequence of $\mathbf{s}$, saying
${\mathbf{s}}_\tau=(1^{m_1},2^{m_2},\ldots,h^{m_h-x})$ where $1\leq h\leq k$ and $0\leq x<m_h$. Then $\delta\in\{m_1,m_2,\ldots,m_{h-1}\}$, also saying $\delta=m_j$ where $j<h\leq k$. Moreover, we have \[m_{k}+m_{k-1}+\cdots+m_{h+1}+x=r\]
and
\begin{align*}
(m_h-x)+m_{h-1}+\cdots+m_{j+1}&=pm_j,
\end{align*}
which implies
\begin{align*}
m_{j+1}+m_{j+2}+\cdots+m_k=pm_j+r.
\end{align*}

Suppose that there are totally $\bar{a}$ blocks before $\overline\Gamma$ in $\tau^r(\lambda)$. Using similar arguments, we can derive that
\begin{align*}
g(\overline\Gamma)=2t-3=2j+2(m_{j+1}+m_{j+2}+\cdots+m_h-x)/m_j-4\bar{a}-1.
\end{align*}
During the process of performing $r$ times of the operation $\tau$, the number of blocks before the first element of $\Gamma$ remains the same. So the number of blocks between $g(\Gamma)$ and $g(\overline\Gamma)$ in $\tau^r(\lambda)$ is $\bar{a}-a$. In addition, since $\overline\Gamma$ is separated from $\Gamma$, we can see that
\begin{align*}
g(\Gamma)&=g(\overline\Gamma)+4(\bar{a}-a)+2\\
&=2j+2(m_{j+1}+m_{j+2}+\cdots+m_h-x)/m_j-4{a}+1\\
&=2j+2(m_{j+1}+m_{j+2}+\cdots+m_h-x+m_j)/m_j-4{a}-1\\
&=2j+2(pm_j+m_j)/m_j-4{a}-1\\
&=2j+2\left\lceil(pm_j+r)/m_j\right\rceil-4{a}-1 \\
&=2j+2\left\lceil(m_{j+1}+m_{j+2}+\cdots+m_k)/m_j\right\rceil-4a-1.
\end{align*}

Case 2. The chain $\Gamma$ includes $1$'s.

In the proof of Theorem \ref{thmflhleq}, we set $\Phi_0=\lambda$ and $\Upsilon_0=\Gamma$. Since $\Upsilon_0$ contains $1$'s in the initial step, and $\varepsilon(\Upsilon_0)=\delta$, we see that $b_o(\Phi_0)=k$ occurs exactly $\delta$ times in the gap-free sequence $\mathbf{s}$. Thus, $\delta=m_k$. Moreover, we have
\begin{align*}
g(\Gamma)=2(k-2a)-1=2k-4a-1.
\end{align*}
We finish the proof.
\end{proof}

\section{Lower bound}\label{seclb}

Assume that $\mathbf{s}=(1^{m_1},2^{m_2},\ldots,k^{m_k})$ is a gap-free sequence, and the weight $j$ determines the lecture hall position. We denote by $\mu$ and $\nu$ the odd partition produced by the sequence $(1^{m_1},2^{m_2},\ldots,(j-1)^{m_{j-1}})$ and $(1^{m_1},2^{m_2},\ldots,j^{m_j})$, respectively. Clearly,  $b_o(\mu)=j-1$ and $b_o(\nu)=j$.

\begin{lemma}\label{thmich}
Suppose that $1$ occurs in $\mu$ and $\Gamma$ is the chain of $\mu$ containing $1$'s. Then we have $\varepsilon(\Gamma)\geq m_{j-1}$.
\end{lemma}
\begin{proof}
Assume $\varepsilon(\Gamma)=r$, and let $t$ be the smallest part which is the maximal element of some block and occurs exactly $r$ times in $\Gamma$.

Suppose that $r<m_{j-1}$. Then the gap-free sequence corresponding to $\tau^{r}(\mu)$ is \[(1^{m_1},2^{m_2},\ldots,(j-1)^{m_{j-1}-r}),\]
which means that $b_o(\tau^r(\mu))=j-1$.

If $j$ is odd, then $j-1$ is even, which forces $\Gamma$ to have the form of
\[\cdots(t+4^{\geq r},t+2^{\geq1})(t^{=r},t-2^{\geq1})(t-4^{>r},t-6^{\geq1})\cdots(3^{>r},1^{\geq1}).\]
We now see that the chain of $\tau^r(\mu)$ containing $1$'s becomes
\[(t-2^{\geq r+1},t-4^{>0})(t-6^{\geq r+1},t-8^{>0})\cdots(5^{\geq r+1},3^{>0})(1^{\geq r+1}),\]
which implies that $b_o(\tau^r(\mu))$ is odd, contradicting that $b_o(\tau^r(\mu))=j-1$ is even.

The proof for the case that $j$ is even is similar and will be omitted.
\end{proof}

\begin{lemma}\label{thmomega}
Let $t$ be the smallest missing odd integer in $\mu$. Then $t$ must lie in the chain $\Gamma$ of $\nu$ containing $1$'s, be the maximal element of some block, and appear exactly $m_j$ times in $\nu$. For other elements $s$ of $\Gamma$ congruent to $t$ modulo $4$, if $s>t$, then $s$ occurs at least $m_j$ times; and if $s<t$, then $s$ occurs more than $m_j$ times.
\end{lemma}
\begin{proof}
If $t=1$, then $b_o(\mu)$ must be even, so $j$ is odd. Then $1$ occurs exactly $m_j$ times in $\nu$. If $3$ occurs in $\nu$, then $3$ occurs in $\mu$ at least $m_j+1$ times. Thus, $5$ occurs at least $m_j$ times in $\nu$. Similarly, if $7$ occurs in $\nu$, then $9$ occurs at least $m_j$ times in $\nu$, and so forth. Thus, $\Gamma$ is either $(1^{m_j})$ or
\begin{align*}
\Gamma=\cdots(9^{\geq m_j},7^{\geq1})(5^{\geq m_j},3^{\geq1})(1^{=m_j}).
\end{align*}

We now assume that $t\neq1$, implying that $1$ occurs in $\mu$. Let $\widetilde\Gamma$ be the chain of $\mu$ containing $1$'s. We have two cases to consider according to the parity of $j$.

If $j$ is even, then $j-1$ is odd, which forces $\widetilde\Gamma$ to be
\[(t-2^{\geq m_{j-1}},t-4^{\geq1})(t-6^{\geq m_{j-1}},t-8^{\geq1})\cdots(5^{\geq m_{j-1}},3^{\geq1})(1^{\geq m_{j-1}}),\]
because of $b_o(\mu)=j-1$ and Lemma \ref{thmich}.

When we perform the operation $\mu\stackrel{+j}{\longrightarrow}$, the block pattern of $\widetilde\Gamma$ is adjusted as
\[(t^{=0},t-2^{\geq m_{j-1}})(t-4^{\geq1},t-6^{\geq m_{j-1}})\cdots(7^{\geq1},5^{\geq m_{j-1}})(3^{\geq1},1^{\geq m_{j-1}}).\]
According to Theorem \ref{thmpm}, we know that $m_j<m_{j-1}$. Therefore,
\begin{align*}
\Gamma=
\cdots(t^{=m_j},t-2^{\geq 1})(t-4^{>m_{j}},t-6^{\geq 1})\cdots(7^{>m_{j}},5^{\geq 1})(3^{>m_{j}},1^{\geq 1}).
\end{align*}
If $t+2$ occurs in $\nu$, then $t+2$ occurs at least $m_j+1$ times in $\mu$. Thus, $t+4$ occurs at least $m_j$ times in $\nu$. Similarly, if $t+6$ occurs in $\nu$, then $t+8$ occurs at least $m_j$ times in $\nu$, and so forth.

The case when $j$ is odd is similar, and the proof is omitted.
\end{proof}

\begin{theorem}\label{thmflhgeq}
If the odd partition $\lambda$ with largest part $\lambda_1$ corresponds to the distinct partition $\pi$ with lecture hall length $\ell_h(\pi)$ under our bijection $\varphi$, then we have $(\lambda_1+1)/2\geq\ell_h(\pi)$.
\end{theorem}
\begin{proof}
Assume that $\mathbf{s}=(1^{m_1},2^{m_2},\ldots,k^{m_k})$ is the gap-free sequence associated with $\lambda$ and $\pi$, and the weight $j$ determines the lecture hall position. Recall that $\mu$ and $\nu$ are the odd partitions produced by the sequence $(1^{m_1},2^{m_2},\ldots,(j-1)^{m_{j-1}})$ and $(1^{m_1},2^{m_2},\ldots,j^{m_j})$, respectively. Obviously, $b_o(\mu)=j-1$ and $b_o(\nu)=j$.

For a chain $\Gamma$, we use $\chi(\Gamma)$ and $\omega(\Gamma)$ to denote the first and last element of $\Gamma$ which is the maximal element of the block containing it and occurs exactly $m_j$ times in $\Gamma$, respectively. Note that $\chi(\Gamma)$ and $\omega(\Gamma)$ may be identical.

Set $\Psi_0=\nu$ and let $\Omega_0$ be the chain of $\nu$ containing $1$'s, the largest element of which is denoted $f_0$. Then $\omega(\Omega_0)$ is the smallest missing odd integer in $\mu$ according to Lemma \ref{thmomega}. Assume that there are $u_0$ and $v_0$ blocks before $\omega(\Omega_0)$ in $\Psi_0$ and $\Omega_0$, respectively. Then $\omega(\Omega_0)=2(j-2u_0)-1$ since $b_o(\nu)=j$, and $f_0=\omega(\Omega_0)+4v_0=2(j-2(u_0-v_0))-1$.

Set $\Psi_1=\eta(\Psi_0)$ where $\eta$ denotes the operation of adding $m_j$ successive weights in $\mathbf{s}$ as we did in the construction of $\varphi^{-1}$. The weights added are in fact all equal to $j+1$ since $m_{j+1}\geq m_j$ (see Remark \ref{remmj}). The block pattern of $\Omega_0$ was changed when we performed $\eta$, resulting in $\chi(\Omega_0)$ must disappear in $\Psi_1$. Thus, $f_0+2$ occurs exactly $m_j$ times in $\Psi_1$ and lies in a chain of odd length. Let $\Omega_1$ be the chain of $\Psi_1$ containing $f_0+2$, the largest part of which is denoted $f_1$. For each block before $f_0+2$, its maximal element occurs at least $m_j$ times; and for each block behind $f_0+2$, its maximal element occurs more than $m_j$ times. Thus, $\omega(\Omega_1)=f_0+2$. Assume that there are $u_1$ and $v_1$ blocks before $\omega(\Omega_1)$ in $\Psi_1$ and $\Omega_1$, respectively. Then we have $u_1=u_0-v_0$. Thus, $\omega(\Omega_1)=2(j-2u_1)+1$ and
\[f_1=\omega(\Omega_1)+4v_1=2(j-2(u_1-v_1))+1.\]

We proceed with the operation $\eta$ in total $q:=\left\lceil(m_{j+1}+\cdots+m_k)/m_j\right\rceil-1$ times, and get two sequences, one of which is formed by the partitions $(\Psi_0,\Psi_1,\ldots,\Psi_q)$, and the other of which is formed by the chains $(\Omega_0,\Omega_1,\ldots,\Omega_q)$ satisfying $\Psi_{i}=\eta(\Psi_{i-1})$ and $\Omega_{i}$ is the chain of $\Psi_{i}$ containing the element $f_{i-1}+2$, where $f_{i-1}$ is the first element of $\Omega_{i-1}$. The block pattern of $\Omega_{i-1}$ will not be destroyed by the parts less than its smallest part before we finish the process of $\eta$ (see Theorem \ref{thminvkey}). This guarantees that
\begin{itemize}
\item $\chi(\Omega_{i-1})$ disappears in $\Psi_i$;
\item $f_{i-1}+2$ occurs exactly $m_j$ times in $\Psi_{i}$;
\item for each block before $f_{i-1}+2$, its maximal element occurs at least $m_j$ times, and for each block
      behind $f_{i-1}+2$, its maximal element occurs more than $m_j$ times;
\item $\omega(\Omega_{i})=f_{i-1}+2$;
\item $\ell(\Omega_i)$ is odd.
\end{itemize}
Assume that there are $u_i$ and $v_i$ blocks before $\omega(\Omega_i)$ in $\Psi_i$ and $\Omega_i$, respectively. Clearly, we have $u_i=u_{i-1}-v_{i-1}$. Thus,
\[\omega(\Omega_i)=f_{i-1}+2=(2(j-2(u_{i-1}-v_{i-1}))+2i-3)+2=2(j-2u_i)+2i-1\]
and $f_i=\omega(\Omega_i)+4v_i=2(j-2(u_i-v_i))+2i-1$.

Since there are $u_q-v_q$ blocks before $f_q$ in $\Psi_q$, the largest part of $\Psi_q$ is at least $f_q+4(u_q-v_q)$.

Let $r_j=(m_{j+1}+\cdots+m_k)-qm_j$. It is routine to verify that $0<r_j\leq m_j$. Now there are $r_j$ weights left in the gap-free sequence after performing the operation $\eta$ in total $q$ times. To finish the inverse of our bijection, we next add these $r_j$ remaining weights of $\mathbf{s}$ to $\Psi_q$ and get the desired odd partition $\lambda$. Although we are unable to compute the exact value of the largest part $\lambda_1$ of $\lambda$, we know that $\lambda_1\geq f_q+4(u_q-v_q)+2$ because the length of $\Omega_q$ is odd and we will perform the weight addition operation at least once. Since $f_q=2(j-2(u_q-v_q))+2q-1$, we conclude that
\[\lambda_1\geq f_q+4(u_q-v_q)+2=2j+2q+1.\]

On the other hand, we have $\pi_j=m_j+m_{j+1}+\cdots+m_k$, and
\begin{align*}
\ell_h(\pi)&=\left\lceil\frac{\pi_j}{m_j}\right\rceil+j-1\\
&=\left\lceil\frac{m_j+m_{j+1}+\cdots+m_k}{m_j}\right\rceil+j-1\\
&=\left\lceil\frac{m_{j+1}+\cdots+m_k}{m_j}\right\rceil+j\\
&=q+1+j.
\end{align*}
Therefore, we arrive at $(\lambda_1+1)/2\geq\ell_h(\pi)$.
\end{proof}

\begin{example}
Let $\mathbf{s}=(1^{11},2^{47},3^{10},4^9,5^{11},6^8,7^{10},8^{14},9^6,{10^5},11^{12},12^6,13^5,14^4)$. It is routine to check that the weight $j=10$ determines the lecture hall position. Thus, we get $m_j=5$ and $q=\lceil(12+6+5+4)/5\rceil-1=5$. We aim to verify that \[\lambda_1\geq 2j+2q+1=31.\]
We will perform totally $q=5$ times of the operation $\eta$. For convenience, let $\mathbf{w}_{i-1}$ denote the collection of $m_j$ weights added to $\Psi_{i-1}$ to obtain $\Psi_i$ for $1\leq i\leq 5$, and we use $\mathbf{w}_5$ to denote the collection of the weights left in $\mathbf{s}$. First, the subsequence $(1^{m_1},2^{m_2},\ldots,j^{m_j})=(1^{11},2^{47},3^{10},4^9,5^{11},6^8,7^{10},8^{14},9^6,{10^5})$ produces \[\Psi_0=\nu=(23^{10},21,17^3,13^5,13^2,11^7,7^5,5,3^{13},1).\]
Second, the weights to be added is $\{11^{12},12^6,13^5,14^4\}$; see Figure \ref{figpsi} for the computation of other $\Psi_i$.
\begin{figure}[ht]
\begin{center}
\begin{tikzpicture}
\draw (0,0)--(15.2,0); \draw (0,0.6)--(15.2,0.6); \draw (0,1.2)--(15.2,1.2); \draw (0,1.8)--(15.2,1.8);
\draw (0,2.4)--(15.2,2.4); \draw (0,3)--(15.2,3); \draw (0,3.6)--(15.2,3.6); \draw (0,4.2)--(15.2,4.2);

\node at (0.15,3.9) {$i$};\node at (0.15,3.3) {$0$};\node at (0.15,2.7) {$1$};\node at (0.15,2.1) {$2$};
\node at (0.15,1.5) {$3$};\node at (0.15,0.9) {$4$};\node at (0.15,0.3) {$5$};

\node at (4.7,0.3) {$(29^5,27^3,25^8,21^2,19^6,15^8,13,11^2,9^3,7^3,5^{10},3^4,1^{10})$};
\node at (4.7,0.9) {$(27^8,25^3,23^5,19^5,17^3,15^3,13^6,9^2,7^4,5^{9},3^5,1^{9})$};
\node at (4.7,1.5) {$(27^3,25^8,21^5,17^8,13^7,11^2,7^3,5^{10},3^4,1^{10})$};
\node at (4.7,2.1) {$(25^9,23^2,19^5,17^3,15^5,13^2,11^7,5^{11},3^3,1^{11})$};
\node at (4.7,2.7) {$(25^4,23^7,17^8,13^7,11^2,9^5,5^{6},3^8,1^{6})$};
\node at (4.7,3.3) {$(23^{10},21,17^3,15^5,13^2,11^7,7^5,5,3^{13},1)$};
\node at (4.7,3.9) {$\Psi_i$};

\node at (11.2,0.3) {$(29^5,27^3)(25^8)$};
\node at (11.2,0.9) {$(27^8,25^3)(23^5)$};
\node at (11.2,1.5) {$(21^5)$};
\node at (11.2,2.1) {$(19^5,17^3)(15^5,13^2)(11^7)$};
\node at (11.2,2.7) {$(13^7,11^2)(9^5)$};
\node at (11.2,3.3) {$(7^5,5)(3^{13},1)$};
\node at (11.2,3.9) {$\Omega_i$};

\node at (14.3,3.9) {$\mathbf{w}_i$};
\node at (14.3,3.3) {$\{11^5\}$};\node at (14.3,2.7) {$\{11^5\}$};\node at (14.3,2.1) {$\{11^2,12^3\}$};
\node at (14.3,1.5) {$\{12^3,13^2\}$};\node at (14.3,0.9) {$\{13^3,14^2\}$};\node at (14.3,0.3) {$\{14^2\}$};

\draw (0,0)--(0,4.2);\draw (0.3,0)--(0.3,4.2);\draw (9.1,0)--(9.1,4.2);\draw (13.4,0)--(13.4,4.2);
\draw (15.2,0)--(15.2,4.2);
\end{tikzpicture}\caption{The computation of $\Psi_i$.}\label{figpsi}
\end{center}
\end{figure}
Finally, we add the remaining $r_j=(12+6+5+4)-5\cdot5=2$ weights (i.e., two $14$'s) to $\Psi_5$ to get \[\lambda=(31^2,29^3,27^5,25^6,21^4,19^4,17,15^8,11^4,9,7^5,5^8,3^6,1^8).\]
We find that $\lambda_1=31$, which is indeed greater than or equal to $2j+2q+1$.
\end{example}

\begin{theorem}\label{thminvkey}
Preserve the notation in the proof of Theorem \ref{thmflhgeq}. The block pattern of $\Omega_i$ will not be affected by the parts less than its smallest part before we finish the operation $\eta$.
\end{theorem}
\begin{proof}
To see why, for sake of contradiction, we assume that the block pattern of $\Omega_i$ was affected when we add $w$ weights of $\mathbf{s}$ to $\Psi_{i}$ where $1\leq w<m_j$. The new chain is denoted $\widetilde{\Omega}_i$, and the new odd partition is denoted $\widetilde{\Psi}_i$. More precisely, $\widetilde{\Omega}_i$ is obtained from $\Omega_i$ by appending either a chain of odd length without $1$'s or the chain containing $1$'s. For the first case, $\widetilde{\Omega}_i$ is a chain of odd length without $1$'s; for the latter case, the $(w+1)$-st weight is different from the $w$-th weight. Note that $f_i+2$ must appear in $\widetilde{\Omega}_i$, and the gap-free sequence associated with $\widetilde{\Psi}_i$ is a subsequence of $\mathbf{s}$, saying $\widetilde{\mathbf{s}}=(1^{m_1},2^{m_2},\ldots,h^{m_h-x})$ where $j\leq h\leq k$ and $0\leq x<m_h$.

We first show that the chains before $\widetilde{\Omega}_i$ are all $m_j$-fixed.

It is routine to verify that there always exists an integer $d$ such that all chains before $\widetilde{\Omega}_i$ are $d$-fixed. We take the largest such $d$, which means that some chains under consideration are not $(d+1)$-fixed.

We claim that $d\geq m_j$.

Assume that $d<m_j$. Let $\Gamma$ be the last chain before $\widetilde{\Omega}_i$ which is not $(d+1)$-fixed. Suppose that there are $u$ and $v$ blocks before $g(\Gamma)$ (i.e., the largest part of $\Gamma$) and $f_i+2$ in $\widetilde{\Psi}_i$, respectively. Clearly, $\Gamma$ excludes $1$'s. We know from Theorem \ref{thmmapkey}
that $d\in\{m_1,m_2,\ldots,m_{h-1}\}$, saying $d=m_r$, and that
\begin{align*}
g(\Gamma)&=2r+2\left\lceil(m_{r+1}+m_{r+2}+\cdots+m_h-x)/m_r\right\rceil-4u-1.
\end{align*}
From the proof of Theorem \ref{thmflhgeq}, we know that
\begin{align*}
f_i+2&=2(j-2(u_i-v_i))+2i+1\\
&=2(j-2v)+2i+1.
\end{align*}

If $\Gamma$ has odd length, then $g(\Gamma)\geq (f_i+2)+4(v-u)$, so
\begin{align*}
r+\left\lceil(m_{r+1}+m_{r+2}+\cdots+m_h-x)/m_r\right\rceil\geq j+i+1.
\end{align*}
From the proof of Theorem \ref{thmmapkey}, we know that $(m_{r+1}+m_{r+2}+\cdots+m_h-x)$ is divided by $m_r$. We now conclude that
\begin{align*}
\left\lceil\frac{\pi_r}{m_r}\right\rceil+r-1
&=\left\lceil\frac{(m_r+m_{r+1}+\cdots+m_h-x)+x+m_{h+1}+\cdots+m_k}{m_r}\right\rceil+r-1\\
&=1+\left\lceil\frac{m_{r+1}+\cdots+m_h-x}{m_r}\right\rceil
+\left\lceil\frac{x+m_{h+1}+\cdots+m_k}{m_r}\right\rceil+r-1\\
&\geq j+i+1+\left\lceil\frac{x+m_{h+1}+\cdots+m_k}{m_r}\right\rceil\\
&>j+i+\left\lceil\frac{x+m_{h+1}+\cdots+m_k}{m_r}\right\rceil.
\end{align*}

If $\Gamma$ has even length, then $g(\Gamma)>(f_i+2)+4(v-u)$, so
\begin{align*}
r+\left\lceil(m_{r+1}+m_{r+2}+\cdots+m_h-x)/m_r\right\rceil>j+i+1.
\end{align*}
According to the proof of Theorem \ref{thmmapkey}, we know that $m_{r+1}+m_{r+2}+\cdots+m_h-x$ is not a multiple of $m_r$, which implies that
\begin{align*}
r+(m_{r+1}+m_{r+2}+\cdots+m_h-x)/m_r>j+i+1.
\end{align*}
We now can assume that $m_{r+1}+m_{r+2}+\cdots+m_h-x=(j+i+1-r)m_r+s$ where $s>0$. We conclude that
\begin{align*}
\left\lceil\frac{\pi_r}{m_r}\right\rceil+r-1
&=\left\lceil\frac{(m_r+m_{r+1}+\cdots+m_h-x)+x+m_{h+1}+\cdots+m_k}{m_r}\right\rceil+r-1\\
&=\left\lceil\frac{(m_r+m_{r+1}+\cdots+m_h-x-s)+s+x+m_{h+1}+\cdots+m_k}{m_r}\right\rceil+r-1\\
&=1+\left\lceil\frac{m_{r+1}+\cdots+m_h-x-s}{m_r}\right\rceil+\left\lceil\frac{s+x+m_{h+1}+\cdots+m_k}{m_r}\right\rceil+r-1\\
&=j+i+1+\left\lceil\frac{s+x+m_{h+1}+\cdots+m_k}{m_r}\right\rceil\\
&>j+i+\left\lceil\frac{s+x+m_{h+1}+\cdots+m_k}{m_r}\right\rceil\\
&\geq j+i+\left\lceil\frac{x+m_{h+1}+\cdots+m_k}{m_r}\right\rceil.
\end{align*}

Thus, we always have
\begin{align*}
\left\lceil\frac{\pi_r}{m_r}\right\rceil+r-1
&>j+i+\left\lceil\frac{x+m_{h+1}+\cdots+m_k}{m_r}\right\rceil\\
&\geq j+i+\left\lceil\frac{x+m_{h+1}+\cdots+m_k}{m_j}\right\rceil,
\end{align*}
by the assumption that $m_r=d<m_j$.

Meanwhile, we have
\begin{align*}
\pi_j&=m_j+im_j+x+m_{h+1}+\cdots+m_k,
\end{align*}
which implies that
\begin{align*}
\left\lceil\frac{\pi_j}{m_j}\right\rceil+j-1
&=j+i+\left\lceil\frac{x+m_{h+1}+\cdots+m_k}{m_j}\right\rceil.
\end{align*}
Since the weight $j$ determines the lecture hall position, we obtain
\begin{align*}
\left\lceil\frac{\pi_j}{m_j}\right\rceil+j-1\geq\left\lceil\frac{\pi_r}{m_r}\right\rceil+r-1,
\end{align*}
a contradiction.

Thus, the claim that $d\geq m_j$ is true. Hence, all chains before $\widetilde{\Omega}_i$ are $d$-fixed, which means that they are also $m_j$-fixed.

We next discuss the properties of the chain $\widetilde{\Omega}_i$.

Obviously, $e:=\varepsilon(\widetilde{\Omega}_i)\leq w$ since $f_i+2$ as the maximal element of the block containing it occurs exactly $w$ times in $\widetilde{\Omega}_i$. Because the chains before $\widetilde{\Omega}_i$ are all $m_j$-fixed, they are also $e$-fixed since $e\leq w<m_j$.

If $\widetilde{\Omega}_i$ is a chain of odd length without $1$'s, then $\widetilde{\Omega}_i$ is not $(e+1)$-fixed otherwise we get $\varepsilon(\widetilde{\Omega}_i)\geq e+1$. It follows from Theorem \ref{thmmapkey} that \[e\in\{m_1,m_2,\ldots,m_{h-1}\}.\]
If $\widetilde{\Omega}_i$ contains $1$'s, then Theorem \ref{thmmapkey} tells us that $e=m_h-x$. For this case, the $w$-th weight added is $h$. Since the $(w+1)$-st weight added is different from the $w$-th weight added, we conclude that $x=0$ and $h<k$. Thus, $e=m_h$. We always have $e\in\{m_1,m_2,\ldots,m_{h}\}$, saying $e=m_l$. It follows from the proof of Theorem \ref{thmmapkey} that $m_l+m_{l+1}+\cdots+m_h-x$ is divided by $m_l$ (note that $x=0$ if $m_l=m_h$), assuming that \[m_l+m_{l+1}+\cdots+m_h-x=cm_l.\]

Obviously, $f_i+2$ is the maximal element of the block containing it in $\widetilde{\Omega}_i$. Assume that there are $b$ blocks before $f_i+2$ in $\widetilde{\Omega}_i$. Thus, $g(\widetilde{\Omega}_i)=(f_i+2)+4b$. Alternatively, it follows from Theorem \ref{thmmapkey} that
\begin{align*}
g(\widetilde{\Omega}_i)&=2l+2(m_{l+1}+m_{l+2}+\cdots+m_h-x)/m_{l}-4(u_i-v_i-b)-1\\
&=2(l-2(u_{i}-v_{i}-b))+2c-3.
\end{align*}
Thus, we obtain
\begin{align*}
f_i+2=2(l-2(u_{i}-v_{i}))+2c-3.
\end{align*}
In addition, we know that $f_i=2(j-2(u_i-v_i))+2i-1$ from the proof of Theorem \ref{thmflhgeq}. So,
\begin{align}\label{eql+c}
l+c=j+i+2.
\end{align}

For convenience, we write
\begin{align*}
x+m_{h+1}+\cdots+m_k&=\alpha_lm_l+\beta_l,\\
w+x+m_{h+1}+\cdots+m_k&=\alpha_jm_j+\beta_j,
\end{align*}
where $1\leq\beta_l\leq m_l$ and $1\leq\beta_j\leq m_j$. We now have
\begin{align*}
\pi_l&=m_l+m_{l+1}+\cdots+m_h+m_{h+1}+\cdots+m_k\\
&=(cm_l+x)+m_{h+1}+\cdots+m_k\\
&\equiv \beta_l \pmod{m_l}
\end{align*}
and
\begin{align*}
\pi_j&=m_j+im_j+w+x+m_{h+1}+\cdots+m_k\\
&=(i+1)m_j+w+x+m_{h+1}+\cdots+m_k\\
&\equiv \beta_j \pmod{m_j}.
\end{align*}
Furthermore, we obtain
\begin{align*}
\left\lceil\frac{\pi_l}{m_l}\right\rceil+l-1&=c+\left\lceil\frac{x+m_{h+1}+\cdots+m_k}{m_l}\right\rceil+l-1\\
&\geq c+l-1+\left\lceil\frac{x+m_{h+1}+\cdots+m_k}{m_j}\right\rceil
\end{align*}
because $m_l=e\leq w<m_j$, and
\begin{align*}
\left\lceil\frac{\pi_j}{m_j}\right\rceil+j-1&=i+j+\left\lceil\frac{w+x+m_{h+1}+\cdots+m_k}{m_j}\right\rceil\\
&=c+l-2+\left\lceil\frac{w+x+m_{h+1}+\cdots+m_k}{m_j}\right\rceil\tag{by \eqref{eql+c}}\\
&\leq c+l-1+\left\lceil\frac{x+m_{h+1}+\cdots+m_k}{m_j}\right\rceil
\end{align*}
because $w<m_j$.

We now get
\begin{align*}
\left\lceil\frac{\pi_l}{m_l}\right\rceil+l-1\geq\left\lceil\frac{\pi_j}{m_j}\right\rceil+j-1.
\end{align*}
Since $\ell_h(\pi)=\lceil{\pi_j}/{m_j}\rceil+j-1$, we have
\[\left\lceil\frac{\pi_l}{m_l}\right\rceil+l-1=\left\lceil\frac{\pi_j}{m_j}\right\rceil+j-1,\]
which forces us to obtain
\begin{align*}
\left\lceil\frac{x+m_{h+1}+\cdots+m_k}{m_l}\right\rceil+1&=\left\lceil\frac{w+x+m_{h+1}+\cdots+m_k}{m_j}\right\rceil,
\end{align*}
or equivalently, $(\alpha_l+1)+1=\alpha_j+1$. So we have $\alpha_l=\alpha_j-1$.

On the other hand, we have
\begin{align*}
\alpha_lm_l+\beta_l&=x+m_{h+1}+\cdots+m_k\\
&=\alpha_jm_j+\beta_j-w\\
&>\alpha_jm_j+\beta_j-m_j\\
&=(\alpha_j-1)m_j+\beta_j\\
&=\alpha_lm_j+\beta_j\\
&>\alpha_lm_l+\beta_j.
\end{align*}
So we conclude that $\beta_l>\beta_j$, contradicting that $\pi_j$ determines the lecture hall position.
\end{proof}

Combining Theorems \ref{thmflhleq} and  \ref{thmflhgeq} together, we get the following result.
\begin{theorem}
Given an odd partition $\lambda$, if $\pi=\varphi(\lambda)$, then $(g(\lambda)+1)/2=\ell_h(\pi)$.
\end{theorem}

\section{Concluding remarks}\label{seccr}

\subsection{Block index}
The block index for ordinary partitions (no restriction on parts) was defined in \cite{LWXi23}, which is very analogous to the block odd index for odd partitions. To understand this notion, we only need to replace ``odd integer"" by ``integer'' in the definition of block odd index. We use $b(\lambda)$ to denote the block index of $\lambda$. For instance, if
\[\lambda=(16^4,9^2,8^3,7,5^6,4^3,3,2,1^5),\]
then it has $6$ blocks, namely $(16^4),(9^2,8^3),(7),(5^6,4^3),(3,2)$ and $(1^5)$. Thus,
\[b(\lambda)=2+2+2+2+2+1=11.\]
Letting $\pi=(14^3,13^2,12^5,7^4,6,4^2,2^5,1^6)$, then the blocks of $\pi$ are $(14^3,13^2)$, $(12^5)$, $(7^4,6)$, $(4^2)$ and $(2^5,1^6)$. So $b(\pi)=2+2+2+2+2=10$.

Wang and Zhou \cite{WZhZ24} found an interesting application of the block index to the newly popularized Schmidt's partition theorem. A distinct partition $(\pi_1,\pi_2,\pi_3,\pi_4,\ldots)$ with $\pi_1+\pi_3+\pi_5+\cdots=n$ is called a Schmidt partition of $n$. For simplicity, we use $|\pi|_o$ to denote the sum $\pi_1+\pi_3+\pi_5+\cdots$. For an ordinary partition $\lambda$, let $m_{o,o}(\lambda)$ and $m_{o,e}(\lambda)$ be the number of distinct odd and even parts of $\lambda$ occurring an odd number of times, respectively.
The main result of \cite{WZhZ24} reads as follows.
\begin{theorem}\label{thmsch}
The number of ordinary partitions of $n$ into $m$ parts with $i$ odd parts occurring an odd number of times, $j$ even parts occurring an odd number of times, and block index $k$ equals the number of Schmidt partitions of $n$ into $k$ parts with alternating sum $m$, $i$ odd indexed odd parts and $j$ even indexed odd parts. Equivalently,
\begin{align*}
\sum_{\lambda\in\mathcal{P}}x^{\ell(\lambda)}y^{m_{o,o}(\lambda)}z^{{m_{o,e}(\lambda)}}w^{b(\lambda)}q^{|\lambda|}
&=\sum_{\pi\in\mathcal{S}}x^{\ell_a(\pi)}y^{n_{o,o}(\pi)}z^{n_{o,e}(\pi)}w^{\ell(\pi)}q^{|\pi|_o},
\end{align*}
where $\mathcal{P}$ and $\mathcal{S}$ denotes the set of ordinary and Schmidt partitions, respectively.
\end{theorem}

We are ready to unveil a relation between Theorem \ref{thmmain} and Theorem \ref{thmsch}. For an ordinary partition $\lambda=(\lambda_1,\lambda_2,\ldots,\lambda_l)$, we can easily transform it to an odd partition, denoted $\bar{\lambda}=(\bar{\lambda}_1,\bar{\lambda}_2,\ldots,\bar{\lambda}_l)$, where $\bar{\lambda}_i=2\lambda_i-1$ for $1\leq i\leq l$. Clearly, we have
\begin{align*}
\ell(\lambda)=\ell(\bar{\lambda}),g(\lambda)=(g(\bar{\lambda})+1)/2,m_{o,o}(\lambda)=o_{1,4}(\bar{\lambda}),
m_{o,e}(\lambda)=o_{3,4}(\bar{\lambda}).
\end{align*}
It is not hard to verify that $b(\lambda)=b_o(\bar{\lambda})$. Assuming that $\pi=\varphi(\bar\lambda)$, we see that
\begin{align*}
|\lambda|=\frac{|\bar{\lambda}|+\ell(\bar{\lambda})}{2}
=\frac{|\pi|+\ell_a(\pi)}{2}=\pi_1+\pi_3+\pi_5+\cdots=|\pi|_o.
\end{align*}
Therefore, with the help of our key bijection $\varphi$, we can not only give a combinatorial proof of Theorem \ref{thmsch} but also strengthen it to
\begin{align*}
\sum_{\lambda\in\mathcal{P}}x^{\ell(\lambda)}y_1^{m_{o,o}(\lambda)}y_2^{{m_{o,e}(\lambda)}}z^{b(\lambda)}
w^{g(\lambda)}q^{|\lambda|}
=\sum_{\pi\in\mathcal{S}}x^{\ell_a(\pi)}y_1^{n_{o,o}(\pi)}y_2^{n_{o,e}(\pi)}z^{\ell(\pi)}w^{\ell_{h}(\pi)}q^{|\pi|_o}.
\end{align*}

Our bijection $\varphi$ can also be applied to prove some other results related to block index or block odd index. We need another two simple bijections, one of which was used to prove Theorem \ref{thmsch} and denoted $\rho$ from now on. Namely, $\rho$ maps an ordinary partition to an odd partition. The other one is denoted $\varpi$, which maps a gap-free partition to an odd partition, and defined below. Assume that $\lambda$ is a gap-free partition with largest part $\lambda_1$ and $\lfloor\lambda_1/2\rfloor=r$. We choose one copy of $2i-1$ for $1\leq i\leq r$ and all even parts to construct the first $r$ parts of $\pi:=\varpi(\lambda)$ as follows: for each $1\leq j\leq r$, \[\pi_j=2(f_j+f_{j+1}+\cdots+f_{r})+(2j-1),\]
where $f_j$ denotes the number of repetitions of $2j$ in $\lambda$. The parts now left in $\lambda$ are all odd, which are regarded as the remaining parts of $\pi$. We now can combine $\varphi$, $\rho$ and $\varpi$ to prove most of the results in \cite{LWXi23} and \cite{LWX23}. We do not plan to state each theorem because of involving more notions and notation, and do not plan to give the details, just providing hints.
\begin{itemize}
\item Theorem 4.3 in \cite{LWX23}: $\varphi$ and some elementary combinatorial analysis;
\item Theorem 5.3 in \cite{LWX23}: $\varphi$, conjugation transformation and $\varpi$;
\item Theorem 1.8 in \cite{LWXi23}: $\rho$, $\varphi$, conjugation transformation, $\varpi$ and $\rho^{-1}$.
\end{itemize}

The block index is a generalization of the newer statistic ``minimal excludant'', defined to be the smallest missing integer in a partition. Similarly, we define the minimal odd excludant to be the smallest missing odd integer in a partition. In particular, the block odd index is a generalization of the minimal odd excludant for an odd partition. Thus, the results in a series of papers \cite{KLW20,WX22,WX23,WZ23,WZhZ23} can be naturally proved, where our bijection $\varphi$ serves as a main ingredient.

\subsection{$k$-block index}

For an integer $k\geq 2$, a partition $\lambda$ is called $k$-regular if it contains no multiples of $k$, and is called $k$-strict if each of its parts appears less than $k$ times. We use $\mathcal{R}_k$ and $\mathcal{S}_k$ to denote the set of $k$-regular and $k$-strict partitions, respectively.

Given a $k$-regular partition $\lambda$, we group its parts into blocks as follows. Starting with the
largest part of $\lambda$, saying $\lambda_1$, put all parts that are greater than or equal to $\lambda_1-k$ into a block. Then we remove the block from $\lambda$ and repeat this process on the remaining parts until we end up with an empty partition. Note that each block contains at most $k$ distinct parts. The weight of a block is defined to be $k$ if its largest element is greater than $k$; otherwise, to be its largest element. We define the $k$-block index of $\lambda$, denoted $b_k(\lambda)$, to be the sum of the weight of its blocks. For example, let $k=4$ and
\[\lambda=(26^3,25^8,22^4,21^5,15^2,14,13^6,11^2,10^7,9^4,7,3^{11},1^5).\]
We group the parts of $\lambda$ into blocks in the following manner
\[\lambda=(26^3,25^8,22^4)(21^5)(15^2,14,13^6,11^2)(10^7,9^4,7)(3^{11},1^5),\]
from which we get $b_4(\lambda)=4+4+4+4+3=19$.

It is easy to see that $k$-block index is a generalization of block odd index.

Given a partition $\pi=(\pi_1,\pi_2,\ldots,\pi_l)\in\mathcal{S}_k$, define
\[\ell_k(\pi)=(\pi_1-\pi_{k})+(\pi_{k+1}-\pi_{2k})+(\pi_{2k+1}-\pi_{3k})+\cdots.\]
Amending our bijection $\varphi$ suitably, we can combinatorially prove the following result
\begin{align*}
\sum_{\lambda\in\mathcal{R}_k}x^{\ell(\lambda)}y^{b_k(\lambda)}q^{|\lambda|}
=\sum_{\pi\in\mathcal{S}_k}x^{\ell_k(\pi)}y^{\ell(\pi)}q^{|\pi|},
\end{align*}
whose proof is left to the reader.

\subsection{Lecture hall position}

Let $\pi=(\pi_1,\pi_2,\ldots,\pi_l)$ be a distinct partition in which $\pi_j$ determines the lecture hall position. Write $r_\pi\equiv\pi_j\pmod {m_j}$ where $m_j=\pi_j-\pi_{j+1}$ and $r_\pi>0$.

There is a substantial amount of numerical evidence to conjecture the following result.
\begin{conjecture}\label{conj}
Suppose that $\lambda$ is an odd partition where the largest part occurs exactly $r$ times. If $\pi=\varphi(\lambda)$, then $r=r_\pi$.
\end{conjecture}

See Figure \ref{figevid} for some evidence of this conjecture.
\begin{figure}[ht]
\begin{center}
\begin{tikzpicture}
\draw (0,0)--(14.4,0); \draw (0,0.6)--(14.4,0.6); \draw (0,1.2)--(14.4,1.2); \draw (0,1.8)--(14.4,1.8);
\draw (0,2.4)--(14.4,2.4); \draw (0,3)--(14.4,3); \draw (0,3.6)--(14.4,3.6); \draw (0,4.2)--(14.4,4.2);
\draw (0,4.8)--(14.4,4.8); \draw (0,5.4)--(14.4,5.4);

\draw (0,0)--(0,5.4);
\node at (2.1,5.1) {$\lambda\in\mathcal{O}$};
\node at (2.1,4.5) {$(9,7^2,5^7,3^2,1^{8})$};
\node at (2.1,3.9) {$(11^2,9,7^2,5^7,3,1^{10})$};
\node at (2.1,3.3) {$(17^3,15^4,13^5,11^3,9^2,7)$};
\node at (2.1,2.7) {$(17^4,15^3,13^6,11^2,9^3)$};
\node at (2.1,2.1) {$(17^5,15^4,11^4,9^2,5^3)$};
\node at (2.1,1.5) {$(19^6,17^2,11^6,3,1)$};
\node at (2.1,0.9) {$(19^7,17^2,9^6,5,3,1)$};
\node at (2.1,0.3) {$(21^8,17^9,7^6,5^3,1)$};

\draw (4.2,0)--(4.2,5.4);
\node at (6.3,5.1) {$\pi=\varphi(\lambda)$};
\node at (6.3,4.5) {$(32,22,13,\underline{4},1)$};
\node at (6.3,3.9) {$(36,26,17,\underline{8},5,1)$};
\node at (6.3,3.3) {$(73,64,\underline{39},33,14,11)$};
\node at (6.3,2.7) {$(74,65,\underline{40},34,15,12)$};
\node at (6.3,2.1) {$(\underline{77},68,34,28,9,6)$};
\node at (6.3,1.5) {$(\underline{78},70,36,30,3,1)$};
\node at (6.3,0.9) {$(\underline{88},79,30,24,5,3,1)$};
\node at (6.3,0.3) {$(97,88,\underline{72},64,33,24,1)$};

\draw (8.4,0)--(8.4,5.4);
\node at (9.6,5.1) {$m_j$};
\node at (9.6,4.5) {$4-1=3$};
\node at (9.6,3.9) {$8-5=3$};
\node at (9.6,3.3) {$39-33=6$};
\node at (9.6,2.7) {$40-34=6$};
\node at (9.6,2.1) {$77-68=9$};
\node at (9.6,1.5) {$78-70=8$};
\node at (9.6,0.9) {$88-79=9$};
\node at (9.6,0.3) {$72-64=8$};

\draw (10.8,0)--(10.8,5.4);
\node at (12.6,5.1) {$r_\pi\equiv\pi_j \pmod {m_j}$};
\node at (12.6,4.5) {$1\equiv4\pmod3$};
\node at (12.6,3.9) {$2\equiv8\pmod3$};
\node at (12.6,3.3) {$3\equiv39\pmod6$};
\node at (12.6,2.7) {$4\equiv40\pmod6$};
\node at (12.6,2.1) {$5\equiv77\pmod9$};
\node at (12.6,1.5) {$6\equiv78\pmod8$};
\node at (12.6,0.9) {$7\equiv88\pmod9$};
\node at (12.6,0.3) {$8\equiv72\pmod8$};

\draw (14.4,0)--(14.4,5.4);

\end{tikzpicture}\caption{Some evidence where the underline indicates the lecture hall position.}\label{figevid}
\end{center}
\end{figure}

If Conjecture \ref{conj} is true, then Theorem \ref{thmmain} can be strengthened further. We think the techniques used in Section \ref{secub} and Section \ref{seclb} can help us confirm Conjecture \ref{conj}, but a more careful analysis is needed. At present, we are unable to achieve this final step.

\end{document}